\documentclass[a4paper,11pt]{amsart}
\usepackage[active]{srcltx}
\usepackage{latexsym,amssymb}
\usepackage{stmaryrd}
\usepackage{enumerate}
\usepackage{stackrel}
\usepackage[T1]{fontenc}
\usepackage[latin1]{inputenc}
\usepackage[curve,color,line]{xypic}
\usepackage{xcolor}
\usepackage{hyperref}
\usepackage{pgf,tikz}
\usetikzlibrary{positioning,automata}

\let\op=\llbracket
\let\cl=\rrbracket
\def\pv#1{\ensuremath{\mathsf{#1}}}
\def\pv#1{\ensuremath{\mathsf{#1}}}
\def\om#1#2{\ensuremath{\Omega_{#1}{\pv{#2}}}}
\def\Om#1#2{\ensuremath{\overline{\Omega}_{#1}{\pv{#2}}}}
\def\oms#1#2{\ensuremath{{\Omega}^\sigma_{#1}{\pv{#2}}}}

\let\cal=\mathcal
\def\Cl#1{\ensuremath{\cal{#1}}}

\newtheorem{Thm}{Theorem}[section]
\newtheorem{Prop}[Thm]{Proposition}
\newtheorem{Lemma}[Thm]{Lemma}
\newtheorem{Cor}[Thm]{Corollary}

\newtheorem{Conjecture}[Thm]{Conjecture}
\theoremstyle{definition}

\begin{document}

\title{Towards a pseudoequational proof theory}

\author{J. Almeida}%
\address{CMUP, Dep.\ Matem\'atica, Faculdade de Ci\^encias,
  Universidade do Porto, Rua do Campo Alegre 687, 4169-007 Porto,
  Portugal} \email{jalmeida@fc.up.pt}

\author{O. Kl\'\i ma}%
\address{Dept.\ of Mathematics and Statistics, Masaryk University,
  Kotl\'a\v rsk\'a 2, 611 37 Brno, Czech Republic}%
\email{klima@math.muni.cz}

\begin{abstract}
  A new scheme for proving pseudoidentities from a given set $\Sigma$
  of pseudoidentities, which is clearly sound, is also shown to be
  complete in many instances, such as when $\Sigma$ defines a locally
  finite variety, a pseudovariety of groups, more generally, of
  completely simple semigroups, or of commutative monoids. Many
  further examples when the scheme is complete are given when $\Sigma$
  defines a pseudovariety \pv V which is $\sigma$-reducible for the
  equation $x=y$, provided $\Sigma$ is enough to prove a basis of
  identities for the variety of $\sigma$-algebras generated by~\pv V.
  This gives ample evidence in support of the conjecture that the
  proof scheme is complete in general.
\end{abstract}


\makeatletter%
\@namedef{subjclassname@2010}{%
  \textup{2010} Mathematics Subject Classification}%
\makeatother



\maketitle


\section{Introduction}
\label{sec:intro}

Pseudovarieties are classes of finite algebras closed under taking
homomorphic images, subalgebras and finite direct products. They have
been studied mostly in the context of finite semigroup theory due to
the connections with automata and formal languages. In the framework
of Eilenberg's correspondence \cite{Eilenberg:1976}, determining
whether a regular language enjoys a suitable property of a certain
kind is converted to the membership problem of its syntactic semigroup
in the corresponding pseudovariety of semigroups. On the other hand,
pseudovarieties are in many respects like the varieties of classical
Universal Algebra, admitting relatively free algebras, albeit in
general not finite, but rather profinite, and thus being defined by
formal equations, where pseudoidentities play the role of identities
\cite{Almeida:1994a}. While there is a natural proof scheme for
identities that is sound and complete, which is sometimes referred as
the completeness theorem for equational logic, the situation in the
theory of pseudovarieties is not so simple.

Indeed, the first author has shown that there is no complete finite
deduction system that is sufficient to prove a given pseudoidentity
from a set of hypotheses (or basis) assuming that all models of the
basis are also models of the pseudoidentity
\cite[Section~3.8]{Almeida:1994a}. Some sort of topological closure
operator seems to be required. While such an operator was also
proposed by the first author (see \cite[Section~3.8]{Almeida:1994a}),
it is very hard to handle and only one instance of its application has
been found so far \cite{Almeida&Shahzamanian:2016}.

The main contribution of this paper is a new approach which consists
in starting with all evaluation consequences of the basis, and
completing them in the same term; then, transfinitely alternating
transitive closure and topological closure. This proof scheme is
clearly sound and, by definition, it is suitable for transfinite
induction proofs. We show that it is complete in many familiar
instances of concrete bases: whenever they define locally finite
varieties (in their algebraic language), pseudovarieties of groups or,
more generally, of completely simple semigroups, or pseudovarieties of
commutative monoids. For the proof of some of these results, we need
to show that concrete pseudoidentities that valid in a pseudovariety
with a given basis are provable from the basis. The main technique in
proving such results is invoking excluded structures.

We also show that the proof scheme is complete when the pseudovariety
defined by the basis $\Sigma$ is $\sigma$-reducible with respect to
the equation $x=y$, where $\sigma$ is an implicit signature such that
the variety of $\sigma$-algebras generated by the pseudovariety
defined by $\Sigma$ admits a basis whose elements can be obtained from
$\Sigma$ by our proof scheme. Combining with several reducibility
results that can be found in the literature, an exercise that by no
means we carry out exhaustively, this provides ample evidence in favor
of the conjecture that our proof scheme is complete in general.

\section{Preliminaries}
\label{sec:prelim}

We recall quickly in this section some basic notions from general
algebra, which also serves to fix some notation. The reader is
referred to \cite{Burris&Sankappanavar:1981} for basic notions on
Universal Algebra and to \cite{Almeida:1994a,Almeida:2003cshort,
  Rhodes&Steinberg:2009qt} for an introduction to pseudovarieties and
profinite structures.

By an \emph{algebraic type} $\tau$ we mean a set of operation symbols,
each of which has an associated finite arity. We assume from hereon
that $\tau$ is finite. An algebra of type
$\tau$ or a \emph{$\tau$-algebra} is a nonempty set endowed with an
interpretation of each operation symbol of the type in question as an
operation of the corresponding arity. In general, we fix an algebraic
type and consider only algebras of that type. The (symbols of)
operations of the type are sometimes called the \emph{basic
  operations}.

By a \emph{topological algebra} we mean an algebra endowed with a
topology such that the interpretations of the basic operations are
continuous functions. Such an algebra is \emph{compact} if so is its
topology. We endow finite algebras with the discrete topology. A
topological algebra $S$ is residually \Cl C for a class \Cl C of
algebras if distinct points in~$S$ may be separated by continuous
homomorphisms into members of~\Cl C.

Given a family $(S_i)_{i\in I}$ of topological algebras, where the
index set is directed, and for each pair $(i,j)$ of indices with $i\ge
j$, a continuous homomorphism $\varphi_{i,j}:S_i\to S_j$ such that
$\varphi_{i,i}$ is the identity mapping on $S_i$ and, for $i\ge j\ge
k$, $\varphi_{j,k}\circ\varphi_{i,j}=\varphi_{i,k}$, we may consider
the \emph{inverse limit} $\varprojlim_{i\in I}S_i$, which may be
described as the subset of the product $\prod_{i\in I}S_i$ consisting
of all families $(s_i)_{i\in I}$ such that each $s_i\in S_i$ and, for
$i\ge j$, $\varphi_{i,j}(s_i)=s_j$. In case each $S_i$ is compact and
the functions $\varphi_{i,j}$ are onto, $\varprojlim_{i\in I}S_i$ is
nonempty and a closed subalgebra of $\prod_{i\in I}S_i$.

Recall that a \emph{pseudovariety} is a (nonempty) class of finite
algebras of a given type that is closed under taking homomorphic
images, subalgebras and finite direct products. Let \pv U be a
pseudovariety. A \emph{pro-\pv U algebra} is an inverse limit of
algebras from~\pv U. In other words, a pro-\pv U algebra is a compact
algebra that is residually \pv U. A \emph{profinite algebra} is a
pro-\pv U algebra for the class \pv U of all finite algebras of the
given type. In case a profinite algebra $S$ is finitely generated as a
topological algebra, meaning that a finitely generated subalgebra is
dense, since the signature is assumed to be finite, there are, up to
isomorphism, only countably many finite homomorphic images of~$S$.
Hence, $S$ embeds in a countable product of finite algebras, which
implies that its topology is metrizable; in particular, the topology
of $S$ is characterized by the convergence of sequences.

Given an arbitrary algebra $S$, one may consider all homomorphisms
$S\to F$ onto algebras from a set of representatives of isomorphism
classes of algebras from a given pseudovariety \pv U. These
homomorphisms form a directed set and so we may consider the inverse
limit $\varprojlim_{S\to F}F$, which is called the \emph{pro-\pv U
  completion} of~$S$ and is denoted $\hat{S}_\pv U$. Note that the
natural mapping $S\to \hat{S}_\pv U$ is injective if and only if
$S$~is residually \pv U. In case \pv U consists of all finite algebras
of the given type, we drop the index \pv U and talk about the
\emph{profinite completion} of~$S$.

By an \emph{alphabet} we simply mean a finite set. Alphabets will
appear as free generating sets of various structures.

By a \emph{term} (of a given algebraic type) we mean a formal
expression in a fixed alphabet constructed using the basic operations
according to their arities. In other words, it is an element of the
free algebra in the (Birkhoff) variety of all algebras of the given type.

A pro-\pv U algebra $S$ is said to be \emph{freely generated by $A$}
if it comes endowed with a function $\iota:A\to S$ such that, for
every function $\varphi:A\to T$ into a pro-\pv U algebra $T$, there is
a unique continuous homomorphism $\hat\varphi:S\to T$ such that the
following diagram commutes:
$$\xymatrix{
  A
  \ar[r]^\iota
  \ar[rd]_\varphi
  &
  S
  \ar@{-->}[d]^{\hat{\varphi}}
  \\
  &
  T
}$$
Such a structure $S$ is unique up to isomorphism of topological
algebras and can be easily shown to be precisely the inverse limit of
all $A$-generated members of~\pv U. It is denoted \Om AU. If $n=|A|$,
then we may sometimes write \Om nU instead of~\Om AU as it is easy to
see that, up to isomorphism of topological algebras, \Om AU only
depends on the cardinality of the set $A$ and not on the set itself.
The subalgebra of~\Om AU generated by $\iota(A)$ is denoted \om AU; it
is the algebra in the variety generated by~\pv U that is freely
generated by the alphabet $A$. Note that \Om AU may also be obtained
as the profinite completion $\widehat{\om AU}$.

An element of~\Om AU is sometimes called a \emph{\pv U-pseudoword};
those that lie in \om AU are said to be \emph{finite} whereas the
remaining \pv U-pseudowords are said to be \emph{infinite}. Every
pseudoword $w\in\Om AU$ has a natural interpretation as an operation of
arity $|A|$ on a pro-\pv U algebra $T$: $|A|$-tuples of elements
of~$T$ may be identified with functions $\varphi:A\to T$ and so the
interpretation of $w$ becomes a function $w_T:T^A\to T$; the
image $w_T(\varphi)$ is defined to be $\hat{\varphi}(w)$, where
$\hat{\varphi}$ is given by the above commutative diagram. Viewed as
operations, pseudowords are sometimes called \emph{implicit
  operations} because their natural interpretations commute with
continuous homomorphisms between pro-\pv U algebras.

By a \emph{\pv U-pseudoidentity} we mean a formal equality $u=v$ with
$u,v\in\Om AU$ for some alphabet $A$. A pseudoidentity $u=v$ holds in
an algebra $T\in\pv U$ if $u_T=v_T$. For a set $\Sigma$ of \pv
U-pseudoidentities the class $\op\Sigma\cl_\pv U$ consists of all
algebras $T\in\pv U$ in which all pseudoidentities from~$\Sigma$ hold.
When the ambient pseudovariety \pv U is understood from the context,
we may also write $\op\Sigma\cl$. This defines a pseudovariety
contained in~\pv U, that is, a \emph{subpseudovariety} of~\pv U, and
Reiterman's theorem \cite{Reiterman:1982} states that every
subpseudovariety of~\pv U is of this form.

By an \emph{implicit signature} (over~\pv U) we mean a set $\sigma$ of
\pv U-pseudowords including those of the form $o(a_1,\ldots,a_n)\in\om
AU$, where $A=\{a_1,\ldots,a_n\}$ is an $n$-letter alphabet and $n$ is
the arity of the basic operation~$o$. By the above, every pro-\pv U
algebra has a natural structure of~$\sigma$-algebra. In particular,
this is the case for the algebras from \pv U which, as
$\sigma$-algebras, generate a variety of $\sigma$-algebras denoted
$\pv U^\sigma$. It is also the case of~\Om AU; the $\sigma$-subalgebra
generated by $\iota(A)$ is denoted \oms AU and it is easily shown to
be the algebra of the variety $\pv U^\sigma$ freely generated by~$A$.

Two pseudovarieties that have received a lot of attention are the
pseudovariety \pv S, of all finite semigroups, and \pv M, of all
finite monoids. For an element $s$ of a finite semigroup $S$, there is
a unique power $s^n$ ($n\ge1$) of $s$ which is an idempotent and is
denoted $s^\omega$. The element $s^{2n-1}$ is then also denoted
$s^{\omega-1}$. In terms of the discrete topology, $s^\omega$ is the
limit of the sequence $(s^{n!})_n$ while $s^{\omega-1}$ is the limit
of the sequence $(s^{n!-1})_n$. It follows that if, instead of taking
$S$ finite we take $S$ to be profinite, the sequences in question
still converge and the notation for the limits is retained.

Consider the semiring $\mathbb{N}$ of all natural numbers (including
zero) and the ring $\mathbb{Z}$ of integers, both under the usual
addition and multiplication. These may be viewed as algebras of type
consisting of two binary operation symbols, of addition and
multiplication. Since both underlying additive monoids are monogenic,
so are there finite homomorphic images. Note also that a finite
monogenic additive monoid $M$ carries a natural structure of semiring
via the natural additive homomorphism $\mathbb{N}\to M$. Moreover, in
case $M$ is a group, such a homomorphism factors through the embedding
$\mathbb{N}\to\mathbb{Z}$ as a ring homomorphism $\mathbb{Z}\to M$. It
follows that the profinite completion of each of $\mathbb{N}$ and
$\mathbb{Z}$ as additive structures carries a multiplication which
makes it isomorphic (as topological algebra) with the semiring
$\hat{\mathbb{N}}$, respectively with the ring $\hat{\mathbb{Z}}$.

Let \pv G and $\pv G_p$ denote, respectively, the pseudovarieties of
all finite groups and of all finite $p$-groups. The above allows to
describe the structure of the monoid \Om1M and of the groups \Om1G and
$\Om 1G_p$. Indeed we know that \Om 1M and \Om1G are respectively the
inverse limits of the finite monogenic monoids and of the finite
cyclic groups, which we know that, as additive algebras, carry
multiplicative structures which make them isomorphic to
$\hat{\mathbb{N}}$ and $\hat{\mathbb{Z}}$, respectively. The
isomorphisms $\hat{\mathbb{N}}\to\Om1M$ and $\hat{\mathbb{Z}}\to\Om1G$
are easy to describe: if the free generator is denoted $x$, they send
each natural $n$ in $\mathbb{N}$ or in $\mathbb{Z}$ to $x^n$,
respectively in~\Om1M and in~\Om1G. For this reason, we also denote,
for each $\alpha$ in $\hat{\mathbb{N}}$ or in $\hat{\mathbb{Z}}$ the
image in~\Om1M or in~\Om1G, respectively, by $x^\alpha$. Thus the
usual laws of exponents hold and $(x^\alpha)^\beta=x^{\alpha\beta}$
can be viewed as a composition of implicit operations.

Similarly, $\Om1G_p$ may be identified with the completion
$\hat{\mathbb{Z}}_{\pv G_p}=\varprojlim_n\mathbb{Z}/p^n\mathbb{Z}$,
which is frequently denoted $\mathbb{Z}_p$. Moreover, thanks to the
Chinese Remainder Theorem, it is easy to see that
$\hat{\mathbb{Z}}\simeq\prod_p\mathbb{Z}_p$, where the index $p$ runs
over all primes. The structure of the ring $\mathbb{Z}_p$ is quite
transparent: it is an integral domain and every ideal is both principal
and closed. It follows that the principal ideals of the
ring~$\hat{\mathbb{Z}}$ are the closed ideals, which in turn are the
finitely generated ideals. In particular, every subset of~$\mathbb{Z}$
has a \emph{greatest common divisor}, which is a generator of the
closed ideal generated by the given set.

Since each of the (semi)rings $\mathbb{N}$ and $\mathbb{Z}$ is
residually finite, and the later is even residually~$\pv G_p$, there
are natural embeddings $\mathbb{N}\to\hat{\mathbb{N}}$,
$\mathbb{Z}\to\hat{\mathbb{Z}}$, and $\mathbb{Z}\to\mathbb{Z}_p$,
which we view as inclusion mappings. It is easy to see that the
invertible elements of $\mathbb{Z}_p$ are those that are not divisible
by~$p$. Hence, the invertible elements of~$\hat{\mathbb{Z}}$ are those
that are not divisible by any prime~$p$.

The semiring $\hat{\mathbb{N}}$ has two idempotents, namely $0$ and
$\omega=\lim n!$. The maximal additive group $H_\omega$ containing
$\omega$ is a closed ideal of~$\hat{\mathbb{N}}$ which is generated by
$\omega+1$. The natural continuous homomorphism
$\pi:\hat{\mathbb{N}}\to\hat{\mathbb{Z}}$, mapping $1$ to $1$, also
maps $\omega+1$ to~$1$ and therefore restricts to an isomorphism
$H_\omega\to\hat{\mathbb{Z}}$. Thus, $\hat{\mathbb{Z}}$ may be
identified with~$H_\omega$, which we do from hereon; it is a retract
of~$\hat{\mathbb{N}}$ under the mapping $\alpha\mapsto\omega+\alpha$.
Also note that $\hat{\mathbb{N}}$ is the disjoint union
of~$\mathbb{N}$ with $\hat{\mathbb{Z}}$.

\section{A proof scheme}
\label{sec:scheme}

Let \pv U be a pseudovariety of a certain finite type of algebras
involving only finitary operations. Let $\Sigma$ be a set of
pseudoidentities $u=v$ with $u$ and~$v$ elements of the free pro-\pv U
algebra over some arbitrary finite alphabet. We seek a complete proof
scheme for pseudoidentities valid in the pseudovariety~$\op\Sigma\cl$,
that is, a deduction system that is capable of deducing from~$\Sigma$
exactly the pseudoidentities valid in~$\op\Sigma\cl$.

By transfinite recursion, we define, for each ordinal $\alpha$, a set
$\Sigma_\alpha$ of pseudoidentities over~$A$ as follows:
\begin{itemize}
\item $\Sigma_0$ consists of
  all pairs of the form
  $$\bigl(\mathbf{t}(\varphi(u),w_1,\ldots,w_n),
  \mathbf{t}(\varphi(v),w_1,\ldots,w_n)\bigr)$$
  such that either $u=v$ or $v=u$ is a pseudoidentity from~$\Sigma$,
  say with $u,v\in\Om BU$, $\varphi:\Om BU\to\Om AU$ is a continuous
  homomorphism, $\mathbf{t}$ is a term (in the algebraic language
  of~\pv U), and $w_i\in\Om AU$ ($i=1,\ldots,n$);
\item $\Sigma_{2\alpha+1}$ is the transitive closure of the binary
  relation $\Sigma_{2\alpha}$;
\item $\Sigma_{2\alpha+2}$ is the topological closure of the relation
  $\Sigma_{2\alpha+1}$ in the space $\Om AU\times\Om AU$;
\item if $\alpha$~is a limit ordinal, then
  $\Sigma_\alpha=\bigcup_{\beta<\alpha}\Sigma_\beta$.
\end{itemize}
Note that, if $\Sigma_{\alpha+2}=\Sigma_{\alpha}$, then
$\Sigma_\alpha$ is both transitive and topologically closed, so that
$\Sigma_\beta=\Sigma_\alpha$ for every ordinal $\beta$ with
$\beta\ge\alpha$. Such a condition must hold for $\alpha$ at most
$\omega_1$, the least uncountable ordinal. Hence, the union
$\tilde{\Sigma}=\bigcup_{\alpha}\Sigma_\alpha$ defines a transitive
closed binary relation on~\Om AU.

Consider a binary relation $\theta$ on~\Om AU. We say that $\theta$~is
\emph{stable} if $(u,v)\in\theta$ implies
$$\bigl(\mathbf{t}(u,w_1,\ldots,w_n),
\mathbf{t}(v,w_1,\ldots,w_n)\bigr)\in\theta$$
for every term $\mathbf{t}$ and $w_1,\ldots,w_n\in\Om AU$. We also say
that $\theta$~is \emph{fully invariant} if, for every continuous
endomorphism $\varphi$ of~$\Om AU$ and $(u,v)\in\theta$, we have
$(\varphi(u),\varphi(v))\in\theta$. A stable equivalence relation is
also called a \emph{congruence}.

The following result establishes the soundness of the above proof
scheme.

\begin{Prop}
  \label{p:gen-provable-inequalities}
  The relation $\tilde{\Sigma}$ is a fully invariant closed congruence
  on~$\Om AU$. For every $(u,v)\in\tilde{\Sigma}$, the pseudoidentity
  $u=v$~is valid in~$\op \Sigma\cl$.
\end{Prop}

\begin{proof}
  Note that $\Sigma_0$ is a reflexive and symmetric binary relation
  on~\Om AU, whence so are all $\Sigma_\alpha$ as well as
  $\tilde{\Sigma}$. We prove by transfinite induction on~$\alpha$ that
  each $\Sigma_\alpha$ is a stable fully invariant binary relation
  on~\Om AU whose elements, viewed as pseudoidentities, are valid in the
  pseudovariety $\op\Sigma\cl$.

  For $\alpha=0$, the claimed properties are immediate from the
  definition of~$\Sigma_0$. Assuming the claim holds for all
  $\alpha<\beta$, then it clearly also holds for $\Sigma_\beta$ in
  case $\beta$~is a limit ordinal. Otherwise, $\beta$ is a successor
  ordinal, and we distinguish the cases where $\beta$ is odd or even.

  In case $\beta$ is odd, that is, it is of the form
  $\beta=2\gamma+1$, $\Sigma_\beta$~is the transitive closure
  of~$\Sigma_{2\gamma}$ which, by the induction hypothesis, is a
  stable fully invariant binary relation on~\Om AU whose elements
  define pseudoidentities valid in the pseudovariety $\op\Sigma\cl$.
  The elements of~$\Sigma_\beta$ are pairs $(u_0,u_n)$ such that there
  exist $u_1,\ldots,u_{n-1}$ with each
  $(u_i,u_{i+1})\in\Sigma_{2\gamma}$ ($i=0,\ldots,n-1$). In
  particular, each pseudoidentity $u_i=u_{i+1}$ is valid
  in~$\op\Sigma\cl$, whence so is~$u_0=u_n$. If $\mathbf{t}$ is a term
  and $w_1,\ldots,w_m$ are elements from $\Om AU$ then, by the
  induction hypothesis, each pair
  $\bigl(\mathbf{t}(u_i,w_1,\ldots,w_m),\
  \mathbf{t}(u_{i+1},w_1,\ldots,w_m)\bigr)$ belongs
  to~$\Sigma_{2\gamma}$, and so the pair
  $\bigl(\mathbf{t}(u_0,w_1,\ldots,w_m),
  \mathbf{t}(u_n,w_1,\ldots,w_m)\bigr)$ also belongs
  to~$\Sigma_\beta$. Finally, if $\varphi$ is a continuous
  endomorphism of~\Om AU, by the induction hypothesis each of the
  pairs $(\varphi(u_i),\varphi(u_{i+1}))$ belongs
  to~$\Sigma_{2\gamma}$, which entails that
  $(\varphi(u_0),\varphi(u_n))$~belongs to~$\Sigma_\beta$.

  Consider next the case where $\beta=2\gamma+2$, that is, a nonzero
  and non-limit even ordinal. Then, since \Om AU is a metric space,
  every element from~$\Sigma_\beta$ is the limit $(u,v)$ of a sequence
  $(u_n,v_n)_n$ of elements from~$\Sigma_{2\gamma+1}$. By the
  induction hypothesis, as a pseudoidentity, every element of the
  sequence is valid in~$\op\Sigma\cl$, whence so is $u=v$ since, under
  an evaluation of the elements of~$A$ in a finite algebra, the
  sequences $(u_n)_n$ and $(v_n)_n$ eventually stabilize, precisely at
  the values of $u$ and~$v$, respectively. Since
  $\Sigma_{2\gamma+1}$~is assumed to be stable, for a term
  $\mathbf{t}$ and $w_1,\ldots,w_m\in\Om AU$, the sequence of pairs
  $\bigl(\mathbf{t}(u_n,w_1,\ldots,w_m),
  \mathbf{t}(v_n,w_1,\ldots,w_m)\bigr)_n$ consists of elements
  of~$\Sigma_{2\gamma+1}$, whence its limit
  $\bigl(\mathbf{t}(u,w_1,\ldots,w_m),
  \mathbf{t}(v,w_1,\ldots,w_m)\bigr)$~belongs
  to~$\Sigma_\beta=\overline{\Sigma_{2\gamma+1}}$. Finally, if
  $\varphi$~is a continuous endomorphism of~\Om AU, then the elements
  of the sequence $(\varphi(u_n),\varphi(v_n))_n$ belong
  to~$\Sigma_{2\gamma+1}$ and, therefore, its limit
  $(\varphi(u),\varphi(v))$ belongs to~$\Sigma_\beta$. This completes
  the transfinite induction.

  In view of the already established initial claim, we know that
  $\tilde{\Sigma}$ is a reflexive symmetric stable fully invariant
  binary relation consisting of pairs $(u,v)$ such that the
  pseudoidentity $u=v$ is valid in~$\op\Sigma\cl$. To complete the
  proof, it remains to recall that we already observed that
  $\tilde{\Sigma}$ is a closed transitive binary relation.
\end{proof}

We call the elements of~$\tilde{\Sigma}$ the pseudoidentities
\emph{provable} from~$\Sigma$. A \emph{proof} consists in a
transfinite sequence of steps in which in step~0 we invoke
pseudoidentities from~$\Sigma$, suitably evaluated in~\Om AU, in which both
sides are plugged in the same place of an arbitrary term, and in later
steps we either use transitivity of equality or take limits, in the
latter two cases from already proved steps, or simply collect together
all pseudoidentities in previous steps. In the last step in the proof we
should have a set of pseudoidentities containing the one to be proved.

An alternative but equivalent, in the sense of capturing the same
proved pseudoidentities, definition of proof would be to take a
transfinite sequence of pseudoidentities in which in each step we allow
one of the pseudoidentities of the above step~0, we take $u=w$ if
there are two previous steps of the form $u=v$ and $v=w$, or we
take $u=v$ provided there is a sequence of earlier steps $(u_n=v_n)_n$
with $u=\lim u_n$ and $v=\lim v_n$. The latter steps are called
\emph{limiting steps}. The last step in such a proof should be the
pseudoidentity to be proved.

If such a proof only involves a finite number of non-limiting steps
and no limiting steps, then we say the proof is \emph{algebraic}. Note
that such a proof exists precisely for the pseudoidentities
of~$\Sigma_1$.

Several examples of proofs in the above general sense can be found in
the literature. In fact, all proofs that a pseudovariety defined by
certain pseudoidentities satisfies a given pseudoidentity that we have been
able to find in the literature seem to be expressible in this form.
This suggests the following general conjecture which amounts to
completeness of our proof scheme.

\begin{Conjecture}
  \label{cj:provable}
  A \pv U-pseudoidentity $u=v$ is provable from a set $\Sigma$ of \pv
  U-pseudoidentities whenever $\op\Sigma\cl$ satisfies $u=v$.
\end{Conjecture}

The above statement is to be interpreted as a logical clause depending
on three parameters: \pv U, $\Sigma$, and $u=v$, where the latter two
determine the first one. Since we have been unable to establish the
conjecture in full generality, particular instances of the conjecture,
where one of the parameters is fixed may be of interest as they
provide evidence towards the conjecture. If a parameter is fixed then
those that are not determined by it are interpreted as being
universally quantified.

We say that a set $\Sigma$ of \pv U-pseudoidentities is \emph{h-strong
  (within~\pv U, or for~\pv U)} if the statement in the conjecture
holds for the given fixed choice of~$\Sigma$ and an arbitrary choice
of the \pv U-pseudoidentity $u=v$. In case $\Sigma=\{u=v\}$ consists
of a single pseudoidentity, we also say that $u=v$ is \emph{h-strong}
if so is~$\Sigma$.

A \pv U-pseudoidentity $u=v$ is said to be \emph{t-strong (within~\pv
  U, or for~\pv U)} if the statement in the conjecture holds for the
given fixed choice of~$u=v$ and an arbitrary choice
of~$\Sigma$.\footnote{The letters h and t in ``h/t-strong'' are meant
  to refer to whether the pseudoidentities appear as hypotheses or
  thesis in the proofs.}

The conjecture also involves an ambient pseudovariety \pv U with
respect to which pseudoidentities are taken. We say that the
pseudovariety \pv U is \emph{strong} if the statement in the
conjecture holds that is it holds for every set $\Sigma\cup\{u=v\}$ of
\pv U-pseudoidentities.

Taking into account the results of~\cite[Section~3.8]{Almeida:1994a},
showing that $\Sigma$ is h-strong is equivalent to showing that
$\tilde{\Sigma}$ is a profinite congruence in the sense
of~\cite[Page~139]{Rhodes&Steinberg:2009qt}, that is, for every finite
alphabet $A$, that $\tilde{\Sigma}$ is a closed congruence on~\Om AU
such that the quotient topological algebra $\Om AU/\tilde{\Sigma}$ is
a profinite algebra. A further equivalent formulation of this property
is that, given any two distinct $\tilde{\Sigma}$-classes, there is a
clopen union of $\tilde{\Sigma}$-classes separating them. The
difficulty in establishing this property in general is to obtain
$\tilde{\Sigma}$-saturation of such clopen sets. In case all
congruence classes are determined by a single class, as in the group
case, this program is much easier to achieve
(cf.~Section~\ref{sec:groups}).

We are able to prove the conjecture in several cases of interest. The
remainder of the paper is concerned with gathering evidence for the
conjecture.

\section{A transfer result}
\label{sec:transfer}

The purpose of this section is to show that it is possible to extend
the validity of the conjecture within a certain ambient pseudovariety
\pv V to a larger pseudovariety \pv U provided that \pv V is defined
within \pv U by a basis satisfying the conjecture. We first note that
the converse is also true without the additional assumption on a basis
of~\pv V.

\begin{Prop}
  \label{p:transfer-down}
  Let \pv U be a pseudovariety and let \pv V be a subpseudovariety
  of~\pv U. Suppose that the set $\Gamma$ of \pv U-pseudoidentities
  defines a subpseudovariety of~\pv V and $\Gamma$ is h-strong
  within~\pv U. Then the set
  $\Gamma'=\{\pi(u)=\pi(v):(u=v)\in\Gamma\}$ is h-strong
  within~\pv V, where $\pi$ stands for the natural continuous
  homomorphism $\Om AU\to\Om AV$.\footnote{Here, we should write
    $\pi_A$ instead of~$\pi$, since all finite sets of variables may
    need to be considered. Nevertheless, since the set $A$ will
    usually be clear from the context, most often we abuse notation by
    writing simply $\pi$ independently of the generating set~$A$.}
\end{Prop}

\begin{proof}
  Let $u,v\in\Om AU$ be such that the \pv V-pseudoidentity
  $\pi(u)=\pi(v)$ holds in the pseudovariety %
  $\op\Gamma'\cl_\pv V=\op\Gamma{\cl_\pv U}$. Then, the \pv
  U-pseudoidentity $u=v$ holds in the pseudovariety $\op\Gamma\cl_\pv
  U$. Since we assume that $\Gamma$ is h-strong within~\pv U, it
  follows that there is a proof of the pseudoidentity~$u=v$
  from~$\Gamma$. Projecting by $\pi$ into \Om AV all steps in such a
  proof, we obtain a proof of~$\pi(u)=\pi(v)$ from~$\Gamma'$; more
  precisely, one may easily prove by induction on the ordinal $\alpha$
  that $(\pi\times\pi)(\Gamma_\alpha)\subseteq\Gamma'_\alpha$. Hence,
  $\Gamma'$ is h-strong within~\pv V.
\end{proof}

Going in the opposite direction is more interesting, but requires an
additional assumption.

\begin{Prop}
  \label{p:transfer-up}
  Let \pv U be a pseudovariety and let \pv V be a subpseudovariety
  of~\pv U. Suppose that \pv V admits a basis $\Sigma$ of \pv
  U-pseudoidentities that is h-strong within~\pv U and consists of
  t-strong pseudoidentities. Let $\Gamma$ be a set of \pv
  U-pseudoidentities such that $\op\Gamma\cl_\pv U\subseteq\pv V$. If
  the set $\Gamma'=\{\pi(u)=\pi(v):(u=v)\in\Gamma\}$ is h-strong
  within~\pv V, where $\pi$ is as above, then the set $\Gamma$ is
  h-strong within~\pv U.
\end{Prop}

\begin{proof}
  Let $u=v$ be a \pv U-pseudoidentity valid in the pseudovariety
  $\op\Gamma\cl_\pv U$. This means that the \pv V-pseudoidentity
  $\pi(u)=\pi(v)$ holds in~$\op\Gamma'\cl_\pv V$ as $\op\Gamma\cl_\pv
  U\subseteq\pv V=\op\Sigma\cl_\pv U$. Since $\Gamma'$ is h-strong
  within~\pv V, $\pi(u)=\pi(v)$ is provable from~$\Gamma'$. We need to
  show that $u=v$ is also provable from~$\Gamma$. We first claim that
  there are $u'$ and $v'$ such that $u'=v'$ is provable from~$\Gamma$,
  $\pi(u')=\pi(u)$, and $\pi(v')=\pi(v)$. Before proving the claim, we
  show how it allows us to conclude the proof of the proposition.
  Since $\Sigma$ is h-strong within~\pv U, it follows that~$\Sigma$
  proves $u'=u$ and $v'=v$. On the other hand, since the pseudovariety
  $\op\Gamma\cl$ is contained in~$\pv V$, it satisfies all
  pseudoidentities from~$\Sigma$ and, as these are assumed to be
  t-strong, they are provable from~$\Gamma$. Hence, $\Gamma$ proves
  $u'=u$, $v'=v$ and, assuming the claim, also $u'=v'$, which entails
  that $\Gamma$ proves $u=v$. Thus, it remains to establish the claim.

  Consider the sets $\Gamma'_\alpha$ defined in
  Section~\ref{sec:scheme}. The proof will be complete once we
  establish the above claim that every pseudoidentity
  in~$\Gamma'_\alpha$ is of the form $\pi(w)=\pi(z)$ for some
  pseudoidentity $w=z$ provable from~$\Gamma$. To prove the claim, we
  proceed by transfinite induction on~$\alpha$.

  In case $\alpha=0$, we have a pseudoidentity of the form
  \begin{equation}
    \label{eq:transfer-1}
    \mathbf{t}\bigl(\varphi(\pi(w)),\pi(s_1),\ldots,\pi(s_n)\bigr)
    =\mathbf{t}\bigl(\varphi(\pi(z)),\pi(s_1),\ldots,\pi(s_n)\bigr),
  \end{equation}
  where $\mathbf{t}$ is a term, $w=z$ is a pseudoidentity
  from~$\Gamma$, say over the set of variables $B$, $\varphi:\Om
  BV\to\Om AV$ is a continuous homomorphism, and $s_1,\ldots,s_n\in\Om
  AU$. By the universal property of relatively free profinite
  algebras, there is a continuous homomorphism $\psi:\Om BU\to\Om AU$
  such that the following diagram commutes:
  $$\xymatrix{
    \Om BU
    \ar[r]^\psi
    \ar[d]^{\pi_B}
    &
    \Om AU
    \ar[d]^{\pi_A}
    \\
    \Om BV
    \ar[r]^\varphi
    &
    \Om AV
  }$$
  It follows that the pseudoidentity~\eqref{eq:transfer-1} is obtained
  from
  $$\mathbf{t}(\psi(w),s_1,\ldots,s_n)
  =\mathbf{t}(\psi(z),s_1,\ldots,s_n)$$
  by applying $\pi$ to each member, which completes the basic step
  $\alpha=0$ of the induction.

  Suppose next that $\alpha=2\beta+1$ and that the pseudoidentities
  $w_i=w_{i+1}$ ($i=0,\ldots,k-1$) belong to~$\Gamma'_{2\beta}$. By
  induction hypothesis, there exist $w_i'',w_{i+1}'\in\Om AU$ such
  that the pseudoidentity $w_i''=w_{i+1}'$ is provable from~$\Gamma$,
  $\pi(w_i')=w_i$, and $\pi(w_{i+1}'')=w_{i+1}$. In particular, we
  have $\pi(w_i')=w_i=\pi(w_i'')$ for $i=1,\ldots,k-1$, and so each
  pseudoidentity $w_i'=w_i''$ is also provable from~$\Gamma$. Hence,
  $w_0=w_k$ is provable from~$\Gamma$.

  For the case $\alpha=2\beta+2$, consider a sequence $(w_n=z_n)_n$ of
  \pv U-pseudoidenti\-ties such that each $\pi(w_n)=\pi(z_n)$ belongs
  to $\Gamma'_{2\beta+1}$ and suppose that $(\pi(w_n)=\pi(z_n))_n$
  converges to $\pi(w)=\pi(z)$. By compactness and continuity
  of~$\pi$, we may as well assume that $(w_n=z_n)_n$ converges
  to~$w=z$. Since each pseudoidentity $w_n=z_n$ is provable
  from~$\Gamma$ by induction hypothesis, $w=z$~is also provable
  from~$\Gamma$.

  Since the case of limit ordinals is trivial, the transfinite
  induction is complete and the claim is established.
\end{proof}

With essentially the same arguments, one may replace the t-strongness
hypothesis by provability of $\Sigma$ from~$\Gamma$, which yields the
following corollary.

\begin{Cor}
  \label{c:transfer-up}
  Let \pv U be a pseudovariety and let \pv V be a subpseudovariety
  of~\pv U. Suppose that \pv V admits a basis of \pv
  U-pseudoidentities $\Sigma$ that is h-strong within~\pv U and that
  \pv V is strong. Then every set of \pv U-pseudoidentities from which
  $\Sigma$ is provable is h-strong within~\pv U.\qed
\end{Cor}

\section{Locally finite sets of identities}
\label{sec:locfin}

Recall that a variety is \emph{locally finite} if all its finitely
generated algebras are finite. We also say that a set $\Sigma$ of
identities is \emph{locally finite} if it defines a locally finite
variety.

\begin{Thm}
  \label{t:locfin-identity}
  Every locally finite set of identities is h-strong.
\end{Thm}

\begin{proof}
  Consider a set $\Sigma$ of identities that is locally finite. Let
  $A$ be a finite set and let $\pi:\Om AU\to\Om A{}\op\Sigma\cl$ be
  the natural continuous homomorphism. For each of the finitely many
  elements $s$ of $\Om A{}\op\Sigma\cl$ choose an element $f(s)$
  of~\om AU such that $\pi(f(s))=s$. Let $\tilde{f}=f\circ\pi$ and
  note that, since $\pi$ is continuous, so is $\tilde{f}$.

  For $u\in\om AU$, consider the identity $u=\tilde{f}(u)$. Since
  $\pi(u)=\pi\circ f\circ\pi(u)=\pi(\tilde{f}(u))$, it
  is valid in the pseudovariety $\op\Sigma\cl$. As the basis
  $\Sigma$ is locally finite, it is also valid in the variety
  $[\Sigma]$. By the completeness theorem for equational logic, it
  follows that the identity $u=\tilde{f}(u)$ is algebraically provable
  from~$\Sigma$.

  Finally, consider an arbitrary pseudoidentity $u=v$ with $u,v\in\Om
  AU$ and suppose that it is valid in~$\op\Sigma\cl$, that is,
  $\pi(u)=\pi(v)$. Let $(u_n)_n$ and $(v_n)_n$ be sequences
  in~\om AU converging respectively to~$u$ and~$v$. Since $\pi$~is
  continuous and $\Om A{}\op\Sigma\cl$ is a discrete space, we may as
  well assume that the sequences $(\pi(u_n))_n$ and
  $(\pi(v_n))_n$ are constant. It follows that each of the
  identities $u_n=v_n$ is valid in~$\op\Sigma\cl$, whence the equality
  $\tilde{f}(u_n)=\tilde{f}(v_n)$ holds. By the preceding paragraph,
  we deduce that the identities $u_n=\tilde{f}(u_n)=v_n$ are
  algebraically provable from~$\Sigma$. Hence, the pseudoidentity
  $u=v$ is, sidewise, the limit of the sequence $(u_n=v_n)_n$ of
  algebraically provable pseudoidentities, which shows that $u=v$ is
  provable from~$\Sigma$, thereby establishing that $\Sigma$ is h-strong.
\end{proof}

For a locally finite set of identities, we may take any basis of a
variety generated by a single finite algebra. A classical example of
locally finite identity for semigroups which is not of this type is
$x^2=x$ (see, for instance, \cite{Howie:1976}).

Along the same lines of the proof of Theorem~\ref{t:locfin-identity}, we may
prove the following result.

\begin{Thm}
  \label{t:locfin-pseudovariety}
  Every locally finite set of identities defines a strong
  pseudovariety.
\end{Thm}

\begin{proof}
  Let \pv V be the pseudovariety defined by a locally finite set
  $\Sigma$ of identities and let \Cl V be the variety defined
  by~$\Sigma$. Let $\Gamma$ be a set of \pv V-(pseudo)identities and
  let $u,v\in\om AV=\Om AV$ be such that the identity $u=v$ is valid
  in~$\op\Gamma\cl$. Since \Cl V~is generated by~\pv V, whence also
  $[\Gamma]$ is generated by~$\op\Gamma\cl$, it follows that the
  variety $[\Gamma]$ also satisfies the identity $u=v$. By the
  completeness theorem for equational logic, we deduce that $u=v$ is
  provable from~$\Gamma$.
\end{proof}

Combining Theorem~\ref{t:locfin-identity}
and~\ref{t:locfin-pseudovariety} with Corollary~\ref{c:transfer-up},
we obtain the following result.

\begin{Cor}
  \label{c:locfin}
  Every set of pseudoidentities which proves a locally finite set of
  identities is h-strong.\qed
\end{Cor}

It would be nice to replace the provability assumption in
Corollary~\ref{c:locfin} by the hypothesis that the given set of
pseudoidentity is locally finite. This would follow from
Corollary~\ref{c:transfer-up} if we could show that every identity is
t-strong, which is a particular case of the conjecture which have not
established in general.

\section{T-strongness}
\label{sec:t-strong}

The main proposition of this section gathers the statement of
t-strongness of several pseudoidentities. Some of these results play a
role in the application of Proposition~\ref{p:transfer-up} in later
sections.

We start with a simple lemma which is used in several points in the
sequel.

\begin{Lemma}
  \label{l:1-var-pseudoidentities}
  If $x^\alpha=x^\beta$ is a nontrivial one-variable \pv
  M-pseudoidentity, then it proves the pseudoidentity
  $x^\alpha=x^{\alpha+\omega}$.
\end{Lemma}

\begin{proof}
  Suppose first that $\beta$ is either infinite or finite and greater
  than~$\alpha$. In this case, we can write the given pseudoidentity
  in the form $x^\alpha=x^\alpha x^{\beta-\alpha}$ which proves
  algebraically $x^\alpha=x^\alpha x^{(\beta-\alpha)n!}$ for every
  positive integer~$n$. Taking the limit as $n$ goes to infinite, we
  deduce that $x^\alpha=x^\beta$ proves $x^\alpha=x^\alpha
  x^{(\beta-\alpha)\omega}=x^{\alpha+\omega}$.

  If the initial assumption on~$\beta$ fails, then it holds when we
  interchange the roles of $\alpha$ and $\beta$. By the above, we
  deduce that, in that case, $x^\alpha=x^\beta$ proves
  $x^\beta=x^{\beta+\omega}$ and, therefore, also
  $x^\alpha=x^\beta=x^\beta x^\omega=x^\alpha
  x^\omega=x^{\alpha+\omega}$.
\end{proof}

The following finite semigroups play a role below:
\begin{itemize}
\item $Sl_2$ stands for the two-element semilattice;
\item for positive integers $m$ and $n$, let $B(m,n)$ denote the
  rectangular band $m\times n$, consisting of the pairs $(i,j)$ with
  $1\le i\le m$ and $1\le j\le n$, where multiplication is described
  by $(i,j)(k,\ell)=(i,\ell)$;
\item for positive integers $m$ and $n$, let %
  $C_{m,n}=\langle a: a^m=a^{m+n}\rangle$ be the monogenic semigroup
  with $m+n-1$ elements and maximal subgroup with $n$~elements;
\item for a positive integer $n$, let $C_n$ be the cyclic group of
  order $n$;
\item $B_2$ is the five-element aperiodic Brandt semigroup, which is
  given by the presentation
  $$\langle a,b; aba=a,bab=b, a^2=b^2=0\rangle$$
  as a semigroup with zero;
  \item $N$ is the semigroup with zero given by the presentation
    $$\langle a,b; a^2=b^2=ba=0\rangle;$$
\item $T$ is the semigroup with zero given by the presentation
  $$\langle e,a; e^2=e, ea=a, ae=0\rangle.$$
\end{itemize}

We say that two sets of \pv U-pseudoidentities are \emph{equivalent}
if every pseudoidentity from each of them is provable from the other
set. Another simple result that is useful below is the following
lemma.

\begin{Lemma}
  \label{l:equivalent}
  Let $\Gamma$ be a set of pseudoidentities and suppose that $\Gamma$
  is equivalent to a single pseudoidentity $\varepsilon$. If each
  pseudoidentity in $\Gamma$ is t-strong, then so is $\varepsilon$.
\end{Lemma}

\begin{proof}
  Let $\Sigma$ be a set of pseudoidentities and suppose that
  $\op\Sigma\cl$ satisfies $\varepsilon$. From the equivalence
  hypothesis, it follows  that $\op\Sigma\cl$ satisfies each
  pseudoidentity $\gamma$ from~$\Gamma$. Since $\gamma$ is t-strong,
  we deduce that $\Sigma$ proves~$\gamma$. By the equivalence
  hypothesis again, we conclude that $\Sigma$ proves~$\varepsilon$.
\end{proof}

We are now ready for the announced proposition.

\begin{Prop}
  \label{p:t-strong}
  Each of the following pseudoidentities is t-strong:
  \begin{enumerate}[(i)]
  \item\label{item:G-t-strong}
    $x^\omega=1$;
  \item\label{item:CR-t-strong}
    $x^{\omega+1}=x$;
  \item\label{item:A-t-strong}
    $x^{\omega+1}=x^\omega$;
  \item\label{item:ER-t-strong}
    $(x^\omega y)^\omega x^\omega=(x^\omega y)^\omega$;
  \item\label{item:EJ-t-strong}
    $(x^\omega y)^\omega=(y x^\omega)^\omega$;
  \item\label{item:DS-t-strong}
    $\bigl((xy)^\omega x(xy)^\omega\bigr)^\omega=(xy)^\omega$;
  \item\label{item:DA-t-strong}
    $(xy)^\omega x(xy)^\omega=(xy)^\omega$;
  \item\label{item:DRG-t-strong}
    $((xy)^\omega x)^\omega=(xy)^\omega$;
  \item\label{item:DG-t-strong}
    $(xy)^\omega=(yx)^\omega$;
  \item\label{item:R-t-strong}
    $(xy)^\omega x=(xy)^\omega$;
  \item\label{item:ZE-t-strong}
    $x^\omega y=yx^\omega$;
  \item\label{item:CS-t-strong}
    $(xy)^\omega x=x$.
  \end{enumerate}
\end{Prop}

\begin{proof}
  For simplicity, we may refer to one of
  (\ref{item:G-t-strong})--(\ref{item:CS-t-strong}) as meaning either
  a part of the proposition or the pseudoidentity in it; which is the
  case should be clear from the context.
  
  Let $\Sigma$ be a set of pseudoidentities. Except in the case
  of~(\ref{item:CS-t-strong}), where we should take \pv
  S-pseudoidentities to get a pseudoidentity that is not equivalent to
  $x^\omega=1$, we assume that the elements of~$\Sigma$ are \pv
  M-pseudoidentities. For each of the pseudoidentities $\varepsilon$
  in the statement of the proposition, we assume that the
  pseudovariety $\op\Sigma\cl$ satisfies $\varepsilon$ and show that
  $\Sigma$ proves~$\varepsilon$.
  
  (\ref{item:G-t-strong}) Since all finite semigroups satisfying
  $\Sigma$ are groups, in particular $Sl_2$ fails some pseudoidentity
  $u=v$ from~$\Sigma$, which means that there is some variable that
  occurs in one of the sides but not in the other. Substituting
  $x^\omega$ for that variable and $1$ for all others, we conclude
  that $x^\omega=1$ is algebraically provable from~$\Sigma$.

  (\ref{item:CR-t-strong}) Since the monoid $C_{2,1}^1$
  fails~(\ref{item:CR-t-strong}), it must fail some pseudoidentity
  $u=v$ from~$\Sigma$ under some suitable evaluation. If such an
  evaluation gives the values $0$ and $1$ for $u$ and $v$, then $u$
  and $v$ do not involve the same variables, so that, substituting
  $x^\omega$ for one variable and $1$ for all others, one gets
  $x^\omega=1$, from which (\ref{item:CR-t-strong})~follows.
  Otherwise, one of the sides, say $u$, is evaluated to $a$ and the
  other to either $0$ or $1$. Substituting $1$ for all variables in
  $u$ that are not evaluated to~$a$ and $x$ for every other variable,
  we obtain a pseudoidentity of the form $x=x^\alpha$, with
  $\alpha\in\hat{\mathbb{N}}\setminus\{1\}$. By
  Lemma~\ref{l:1-var-pseudoidentities}, each such pseudoidentity
  proves~(\ref{item:CR-t-strong}).
  
  (\ref{item:A-t-strong}) Each of the cyclic groups of prime order
  fails the pseudoidentity (\ref{item:A-t-strong}). For a pseudoword
  $w$ and a variable $x$, let $w_x$ be the pseudoword that is obtained
  from $w$ by substituting $1$ for every variable except~$x$. Since
  the cyclic group $C_p$ satisfies a pseudoidentity $u=v$ if and only
  if, for every variable $x$, it satisfies the pseudoidentity
  $u_x=v_x$, it follows that, if $C_p$ fails~$\Sigma$, then there
  is $\alpha_p\in\hat{\mathbb{N}}\setminus\mathbb{N}=\hat{\mathbb{Z}}$
  such that $p$ does not divide~$\alpha_p$ and $\Sigma$ proves
  $x^{\alpha_p}=x^\omega$.

  Let $\alpha$ be a greatest common divisor of the $\alpha_p$, which
  is a limit of a sequence of linear combinations of the $\alpha_p$.
  Since each prime $p$~does not divide $\alpha_p$, it cannot
  divide~$\alpha$. Hence, $\alpha$~is invertible in the
  ring~$\hat{\mathbb{Z}}$ and there exists $\beta\in\hat{\mathbb{Z}}$
  such that $\alpha\beta=\omega+1$. Now, if
  $\gamma=\sum_{i=1}^n\gamma_i\alpha_{p_i}$, then $\Sigma$ proves
  $x^\gamma=x^\omega$ by raising both sides of each pseudoidentity
  $x^{\alpha_{p_i}}=x^\omega$ to the $\gamma_i$ power and multiplying
  the results side by side. The pseudoidentity $x^\alpha=x^\omega$ is
  therefore a limit of pseudoidentities provable from~$\Sigma$, and
  whence it is provable from~$\Sigma$. Finally, raising both sides of
  $x^\alpha=x^\omega$ to the $\beta$ power, we deduce that the
  pseudoidentity (\ref{item:A-t-strong}) is provable from~$\Sigma$.

  (\ref{item:ER-t-strong}) As it is easy to see, and well known,
  $B(1,2)^1$ fails a pseudoidentity $u=v$ if and only if, from right
  to left, the order of first occurrences of variables in $u$ and $v$
  is not the same. In particular, the monoid $B(1,2)^1$
  fails~(\ref{item:ER-t-strong}) and, therefore, it fails some
  pseudoidentity from~$\Sigma$. Hence, there is a substitution that
  sends all variables but two to~$1$ that yields from some
  pseudoidentity in~$\Sigma$ a two-variable pseudoidentity of the form
  $ux=vy$, where $x$ and $y$ are distinct variables, or a nontrivial
  pseudoidentity $u=1$ or $1=u$. In the latter case, by substituting
  all the variables by $x^\omega$, we conclude that $\Sigma$ proves
  $x^\omega=1$ and, hence, every pseudoidentity in which both sides
  are products of $\omega$ powers, as is the case
  of~(\ref{item:ER-t-strong}). Thus, it remains to consider the former
  case. Applying the substitution $x\mapsto x^\omega$ and $y\mapsto
  x^\omega y$ and raising both sides to the $\omega$ power, we
  conclude that $\Sigma$ proves the
  pseudoidentity~(\ref{item:ER-t-strong}).

  (\ref{item:EJ-t-strong}) Noting that the sets of pseudoidentities
  $\{(x^\omega y)^\omega x^\omega=(x^\omega y)^\omega, %
  x^\omega(yx^\omega)^\omega=(yx^\omega)^\omega\}$ and %
  $\{(x^\omega y)^\omega=(yx^\omega)^\omega\}$ %
  are equivalent, in view of~(\ref{item:ER-t-strong}) and its dual, it
  suffices to apply Lemma~\ref{l:equivalent}.

  (\ref{item:DS-t-strong}) The evaluation of $x$ to $a$ and $y$ to~$b$
  shows that the monoid $B_2^1$ fails the
  pseudoidentity~(\ref{item:DS-t-strong}). Hence, there is some
  pseudoidentity $u=v$ from~$\Sigma$ and an evaluation of the
  variables in~$B_2^1$ which yields different values for $u$ and~$v$.
  A variable being assigned the value $1$ corresponds to deleting that
  variable. For all other values in~$B_2^1$, since $B_2$ is generated
  by $\{a,b\}$ as a semigroup, we may first substitute the variable by
  a word in the variables $x$ and $y$ and then evaluate $x$ to~$a$ and
  $y$ to~$b$. Thus, under the assumption that $B_2^1$ fails $u=v$, we
  conclude that there is a pseudoidentity $u'=v'$ in $x$ and $y$ that
  can be proved from~$\Sigma$ and which fails in~$B_2^1$ under the
  evaluation $x\mapsto a$ and $y\mapsto b$. Under such an evaluation,
  not both sides are evaluated to~$0$. If one of the sides is~$1$
  then, substituting $x$ for~$y$, we get a pseudoidentity of the form
  $x^\alpha=1$, which proves $x^\omega=1$ and, therefore, also the
  pseudoidentity~(\ref{item:DS-t-strong}). Otherwise, one of the sides
  of the pseudoidentity $u'=v'$, say $u'$, must be a factor of
  $(xy)^\omega$ while $v'$ either admits $x^2$ or $y^2$ as a factor or
  does not start or end with the same letter as~$u'$. By multiplying
  both sides of $u'=v'$ by suitable factors of $(xy)^\omega$, we may
  prove from~$u'=v'$ a pseudoidentity $u''=v''$ of the form
  $(xy)^\omega=w$ where $w$~is a pseudoword that admits at least one
  of the words $x^2$ and $y^2$ as a factor, starts with $x$ and ends
  with~$y$. Substituting $(xy)^\omega x$ for $x$ and $y(xy)^\omega$
  for~$y$, we obtain from $u''=v''$ a pseudoidentity of the form %
  \begin{equation}
    \label{eq:subpseudovarieties-DS-prove-DS-1}
    (xy)^\omega=(xy)^\alpha x(xy)^\beta w
  \end{equation}
  or of the form
  \begin{equation}
    \label{eq:subpseudovarieties-DS-prove-DS-2}
    (xy)^\omega=(xy)^\alpha y(xy)^\beta w,
  \end{equation}
  where $\alpha$ and $\beta$ are infinite exponents and $w$ is some
  pseudoword. From the pseudoidentity
  \eqref{eq:subpseudovarieties-DS-prove-DS-1}, we may prove,
  algebraically,
  $$(xy)^\omega
  =(xy)^\alpha x\cdot(xy)^\omega \cdot (xy)^\beta w %
  =\cdots %
  =\bigl((xy)^\alpha x\bigr)^{n!} (xy)^\omega \bigl((xy)^\beta w\bigr)^{n!}
  $$
  and so, taking limits, also
  $$(xy)^\omega
  =\bigl((xy)^\alpha x\bigr)^\omega (xy)^\omega \bigl((xy)^\beta
  w\bigr)^\omega
  $$
  which entails %
  $(xy)^\omega=\bigl((xy)^\alpha x(xy)^\omega\bigr)^\omega$.
  Similarly, from
  $$(xy)^\omega
  =\bigl((xy)^\alpha x(xy)^\omega\bigr)^{\omega-1}(xy)^\alpha %
  \cdot(xy)^\omega\cdot x(xy)^\omega,$$%
  concentrating on the rightmost factor, we may deduce the desired
  pseudoidentity (\ref{item:DS-t-strong}). If we start from the
  pseudoidentity \eqref{eq:subpseudovarieties-DS-prove-DS-2} instead
  of~\eqref{eq:subpseudovarieties-DS-prove-DS-1}, we reach similarly
  the pseudoidentity %
  $(xy)^\omega=\bigl((xy)^\omega y(xy)^\omega\bigr)^\omega$.
  Interchanging $x$ and $y$, we obtain %
  $(yx)^\omega=\bigl((yx)^\omega x(yx)^\omega\bigr)^\omega$. %
  Multiplying both sides on the left by $x$ and on the right by
  $y(xy)^{\omega-1}$, and rearranging the right hand side (using
  equalities valid in~\pv M) yields the
  pseudoidentity~(\ref{item:DS-t-strong}).

  (\ref{item:DA-t-strong}) First note that (\ref{item:DA-t-strong})
  proves (\ref{item:DS-t-strong}) by simply raising both sides to the
  $\omega$ power; it also proves (\ref{item:A-t-strong}) by
  substituting $y$ by $x$ and using equalities that are valid in~\pv
  M.
  Conversely,
  from~(\ref{item:DS-t-strong}), using (\ref{item:A-t-strong}), we
  deduce that $(xy)^\omega x(xy)^\omega=\bigl((xy)^\omega
  x(xy)^\omega\bigr)^{\omega+1}=\bigl((xy)^\omega
  y(xy)^\omega\bigr)^\omega=(xy)^\omega$. It remains to apply
  Lemma~\ref{l:equivalent} and the above.

  (\ref{item:DRG-t-strong}) The substitution $x\mapsto x^\omega$
  transforms~(\ref{item:DRG-t-strong}) into the pseudoidentity %
  $\bigl((x^\omega y)^\omega x^\omega\bigr)^\omega=(xy)^\omega$. Since
  $(x^\omega y)^\omega x^\omega$ is an idempotent, this shows that
  (\ref{item:DRG-t-strong}) proves~(\ref{item:ER-t-strong}).
  Similarly, upon multiplication of both sides
  of~(\ref{item:DRG-t-strong}) on the right by $(xy)^\omega$, we
  conclude that (\ref{item:DRG-t-strong})
  proves~(\ref{item:DS-t-strong}). In view of the above and
  Lemma~\ref{l:equivalent}, it remains to establish that, together,
  the pseudoidentities (\ref{item:ER-t-strong})
  and~(\ref{item:DS-t-strong}) prove (\ref{item:DRG-t-strong}).
  Indeed, the substitution $x\mapsto(xy)^\omega x$,
  $y\mapsto(xy)^\omega$ in~(\ref{item:ER-t-strong}) gives
  $\bigl((xy)^\omega x\bigr)^\omega=\bigl((xy)^\omega
  x(xy)^\omega\bigr)^\omega$. Combining the latter pseudoidentity
  with~(\ref{item:DS-t-strong}), by transitivity we
  obtain~(\ref{item:DRG-t-strong}).

  (\ref{item:DG-t-strong}) Note first that (\ref{item:DG-t-strong})
  proves the pseudoidentities
  $$(xy)^\omega=(yx)^\omega=y(xy)^{\omega-1}\cdot(xy)^\omega\cdot(xy)^\omega x.$$
  Hence, (\ref{item:DG-t-strong}) proves %
  $(xy)^\omega=\bigl(y(xy)^{\omega-1}\bigr)^{n!}\cdot(xy)^\omega\cdot\bigl((xy)^\omega
  x\bigr)^{n!}$; by taking limits, we get
  $(xy)^\omega=\bigl(y(xy)^{\omega-1}\bigr)^\omega\cdot(xy)^\omega\cdot\bigl((xy)^\omega
  x\bigr)^\omega$ and, therefore, (\ref{item:DG-t-strong})
  proves~(\ref{item:DRG-t-strong}). Since (\ref{item:DG-t-strong}) is
  its own left right dual, (\ref{item:DG-t-strong}) also proves the
  dual of~(\ref{item:DRG-t-strong}), namely the pseudoidentity
  $\bigl(x(yx)^\omega\bigr)^\omega=(yx)^\omega$. Conversely,
  from~(\ref{item:DRG-t-strong}) and its dual, we may prove
  $$(xy)^\omega
  =\bigl(y(xy)^\omega\bigr)^\omega
  =\bigl((yx)^\omega y\bigr)^\omega
  =(yx)^\omega.$$
  Taking into account previous parts of the proposition, it suffices
  to invoke Lemma~\ref{l:equivalent}.

  (\ref{item:R-t-strong}) Substituting $y$ by $x$ in the
  pseudoidentity (\ref{item:R-t-strong}) yields
  (\ref{item:A-t-strong}) while, raising both sides to the $\omega$
  power we obtain~(\ref{item:DRG-t-strong}). Once again, in view of
  Lemma~\ref{l:equivalent}, it suffices to show that, together,
  (\ref{item:A-t-strong}) and~(\ref{item:DRG-t-strong}) also
  prove~(\ref{item:R-t-strong}) which can be done as in proof
  of~(\ref{item:DA-t-strong}):
  $$(xy)^\omega %
  =\bigl((xy)^\omega x\bigr)^\omega %
  =\bigl((xy)^\omega x\bigr)^\omega (xy)^\omega x %
  =(xy)^\omega (xy)^\omega x %
  =(xy)^\omega x.
  $$

  (\ref{item:ZE-t-strong}) Raising both sides
  of~(\ref{item:ZE-t-strong}) to the $\omega$ power, we
  obtain~(\ref{item:EJ-t-strong}) and so $\Sigma$
  proves~(\ref{item:EJ-t-strong}). On the other hand, the monoid $T^1$
  and its left right dual fails the
  pseudoidentity~(\ref{item:ZE-t-strong}). Hence, each of them fails
  some pseudoidentity from~$\Sigma$.

  Let us consider first the fact that $T^1$ fails some pseudoidentity
  from~$\Sigma$. As in earlier arguments, we deduce that $\Sigma$
  proves either the pseudoidentity $x=x^{\omega+1}$ or a two-variable
  pseudoidentity of the form $x^\omega y=w$, where $w$~is a pseudoword
  that admits $yx$ as a factor. If $y$ occurs more than once in~$w$
  then, substituting $x$ by~$1$, we get a nontrivial pseudoidentity
  $y=y^\alpha$, which entails $x=x^{\omega+1}$ by
  Lemma~\ref{l:1-var-pseudoidentities}. Hence, $\Sigma$ proves either
  $x=x^{\omega+1}$ or %
  $x^\omega y=x^\omega yx^\omega$. Working instead with the dual of
  the monoid~$T^1$, we deduce dually that $\Sigma$ proves either
  $x=x^{\omega+1}$ or %
  $yx^\omega=x^\omega yx^\omega$.

  Suppose first that $\Sigma$ proves $x=x^{\omega+1}$. Since $\Sigma$
  also proves~(\ref{item:EJ-t-strong}), it proves the following
  pseudoidentities:
  $$x^\omega y
  =(x^\omega y)^{\omega+1}
  =x^\omega(yx^\omega)^\omega y
  =x^\omega(x^\omega y)^\omega y
  =(x^\omega y)^\omega y
  =(y x^\omega)^\omega y.$$
  Dually, $\Sigma$ proves $yx^\omega=y(x^\omega y)^\omega$ and so
  also~(\ref{item:ZE-t-strong}). Thus, we may assume that
  $\Sigma$~does not prove $x=x^{\omega+1}$. From the above, it follows
  that $\Sigma$ proves $x^\omega y=x^\omega yx^\omega=yx^\omega$, as
  required.

  (\ref{item:CS-t-strong}) Substituting $y$ by~$x$
  in~(\ref{item:CS-t-strong}), we obtain the pseudoidentity
  (\ref{item:CR-t-strong}). On the other hand, since $Sl_2$ fails
  (\ref{item:CS-t-strong}), it also fails some pseudoidentity
  from~$\Sigma$, and so there is some variable $z$ that occurs only on
  one side of that pseudoidentity. Substituting %
  $x^\omega yx^\omega$ for $z$ and $x^\omega$ for every other
  variable, we conclude that $\Sigma$ proves a pseudoidentity of the
  form %
  $(x^\omega yx^\omega)^\alpha=x^\omega$ and so also the special case
  where $\alpha=\omega$, which may be written in the form %
  $(x^\omega y)^\omega x^\omega=x^\omega$. Substituting $xy$ for $y$,
  multiplying both sides on the right by $x$, and using additionally
  the pseudoidentity $x^{\omega+1}=x$, we obtain the required
  pseudoidentity (\ref{item:CS-t-strong}).
\end{proof}

The choice of the monoids and semigroups considered in the proof of
Proposition~\ref{p:t-strong} was guided by several results in the
literature, even though such results are not explicitly used in the
proof. For a pseudoidentity $\varepsilon$, we essentially take a
\emph{complete set of excluded monoids}, that is, a set of finite
monoids such that a pseudovariety \pv V satisfies $\varepsilon$ if and
only if contains none of the monoids from the set. Such sets can be
found in the literature for several pseudoidentities. See
\cite{Almeida:1994a} for further details. It should be noted that the
proof of Proposition~\ref{p:t-strong} in fact implies that the sets in
question are complete sets of excluded monoids.

\section{H-strongness: the role of reducibility}
\label{sec:reducible}

In this section, we consider a method that allows us to give a class
of examples of h-strong sets of pseudoidentities. In all of them, the
key property of the pseudovariety $\pv V=\op\Sigma\cl$ is that every
pseudoidentity $u=v$ valid in~\pv V is the limit of a sequence of
identities in a suitable implicit signature $\sigma$ which are also
valid in~\pv V. In this case we say that the pseudovariety \pv V is
\emph{$\sigma$-reducible for the equation $x=y$}.

In terms of formal proofs, the above property reduces the conjecture
for~$\Sigma$ to proving from $\Sigma$ the $\sigma$-identities which
are valid in~\pv V. If, additionally, one knows a basis $\Sigma'$ of
identities for the variety $\pv V^\sigma$, all that is needed is to
show that $\Sigma'$ is provable from~$\Sigma$. In summary, we have the
following simple observation which provides a method of proof that
underlies all the examples in the remainder of the section.

\begin{Prop}
  \label{p:general}
  Let $\sigma$ be an implicit signature and let $\Sigma$ be a set of
  \pv U-pseudoidentities defining a $\sigma$-reducible pseudovariety
  \pv V for the equation $x=y$. If the variety $\pv V^\sigma$ admits a
  basis whose identities are provable from~$\Sigma$, then $\Sigma$~is
  h-strong.\qed
\end{Prop}

We apply below Proposition~\ref{p:general} to the usual bases of
pseudoidentities of several extensively studied pseudovarieties of
semigroups. Before doing so, it is worth explaining why we use the
above terminology introduced in~\cite{Almeida&Steinberg:2000a} that
has been widely adopted in the literature. This connection needs to be
clarified in order to justify invoking several published results that
are required in our application of Proposition~\ref{p:general}.

As above, fix an ambient pseudovariety \pv U. We also consider a
subpseudovariety \pv V of~\pv U. By a \emph{system of equations} we
mean a set $\Cl S$ whose elements $u=v$ are formal equalities (that
is, pairs) of terms in the algebraic signature of~\pv U. Let $X$ be
the set of variables that occur in~$\Cl S$, a set which, for simplicity,
we assume to be finite. We are interested in \emph{\pv V-solutions}
of~$\Cl S$ on a fixed but arbitrary finite set of generators $A$, which
consist of a mapping assigning to each variable $x\in X$ an element
$\varphi(x)$ of~\Om AU whose natural extension $\hat{\varphi}$ to
terms is such that \pv V satisfies each pseudoidentity
$\hat{\varphi}(u)=\hat{\varphi}(v)$ with $u=v$ a member from~$\Cl S$. Such
systems are often constrained by assigning to each variable $x\in X$ a
clopen subset $K_x$ of~\Om AU. The \pv V-solution $\varphi$ is said to
\emph{satisfy} the constraints if $\varphi(x)\in K_x$ for every $x\in
X$.

The key property of the pseudovariety \pv V introduced
in~\cite{Almeida&Steinberg:2000a} is the following. Let $\sigma$ be an
implicit signature, each of whose elements belongs to some \Om BU,
where $B$~is a finite set. We say that \pv V is
\emph{$\sigma$-reducible for~$\Cl S$} if, for every choice of constraints
for~$X$, if there is a \pv V-solution of~$\Cl S$ satisfying the
constraints, then there is such a solution taking its values in \oms
AU, which we call a \emph{$(\pv V,\sigma)$-solution}. Since the clopen
sets form a basis of the topology of~\Om AU, one immediately
recognizes the following result which is nothing but a topological
reformulation of the definition of reducibility. For that purpose, we
view solutions of the system $\Cl S$ as elements of the product space
$(\Om AU)^X$.

\begin{Prop}
  \label{p:reducibility}
  The pseudovariety \pv V is $\sigma$-reducible for the system of
  equations $\Cl S$ over the set of variables~$X$ if and only if the set
  of $(\pv V,\sigma)$-solutions is dense in the set of all \pv
  V-solutions.\qed
\end{Prop}

In particular, the terminology adopted at the beginning of the section
is consistent with that from~\cite{Almeida&Steinberg:2000a}.

The property of $\sigma$-reducibility of a pseudovariety \pv V was
conceived as part of a strong form of decidability called
\emph{$\sigma$-tameness}. The remaining requirements for
$\sigma$-tameness are computability assumptions, namely: the
pseudovariety \pv V is assumed to be recursively enumerable, the
signature $\sigma$ also recursively enumerable and to consist of
operations that are computable in elements of~\pv V, and the word
problem in \oms AV is supposed to be decidable. The computability
assumptions on the pseudovariety \pv V and the implicit signatures
considered in all our examples are immediately verified.

It should be pointed that the word problem in~\oms AV is equivalent to
the variety $\pv V^\sigma$ to admit a recursive basis of identities.
Besides $\sigma$-reducibility for the identity $x=y$, it is the
knowledge of such a basis that underlies all our applications of
Proposition~\ref{p:general}. Thus, our results rely more properly on
$\sigma$-tameness of~\pv V rather than just $\sigma$-reducibility for
the identity $x=y$.

We are now ready to present our concrete examples of evidence for the
conjecture obtained as applications of Proposition~\ref{p:general}. In
the following, the ambient pseudovariety \pv U will be either the
pseudovariety \pv S of all finite semigroups or the pseudovariety \pv
M of all finite monoids. The implicit signatures involved are often
either $\kappa=\{\_\cdot\_,\_\vphantom{|}^{\omega-1}\}$ or
$\{\_\cdot\_,\_\vphantom{|}^\omega\}$ the latter of which, by abuse of
notation, we also denote~$\omega$.

\subsection{Some simple examples}
\label{sec:A+R+LSl}

Our first example is given by the pseudoidentity
$x^{\omega+1}=x^\omega$, which defines the pseudovariety \pv A of all
finite aperiodic monoids. In view of Sch\"utzenberger's
characterization of star-free languages \cite{Schutzenberger:1965},
this is a very important pseudovariety.

\begin{Thm}
  \label{t:A}
  For $\pv U=\pv M$, the pseudoidentity $x^{\omega+1}=x^\omega$
  is h-strong.
\end{Thm}

\begin{proof}
  The first key ingredient here is that the pseudovariety \pv A is
  $\omega$-reducible for the equation $x=y$. This is proved
  in~\cite[Corollary~3.2]{Almeida&Costa&Zeitoun:2015b} based on
  Henckell's algorithm for the computation of \pv A-pointlike sets of
  finite monoids \cite{Henckell:1988,Henckell&Rhodes&Steinberg:2010}.

  The second key ingredient is a basis of identities for the variety
  $\pv A^\omega$ obtained by
  McCammond~\cite{McCammond:1999a,Almeida&Costa&Zeitoun:2015}. It
  consists of the following identities:
  \begin{align*}
    &(xy)z=x(yz),\ x1=1x=x\\
    &(x^\omega)^\omega
      =(x^r)^\omega
      =x^\omega x^\omega
      =x^\omega\ (r\ge2)\\
    &(xy)^\omega x=x(yx)^\omega\\
    &x^\omega x=xx^\omega=x^\omega
  \end{align*}
  Except for the identities in the last line, which are immediately
  provable from $x^{\omega+1}=x^\omega$, all the other identities are
  valid in all finite monoids and so they require no proof in our
  proof setup within the ambient pseudovariety~\pv M. Applying
  Proposition~\ref{p:general}, we conclude that the pseudoidentity
  $x^{\omega+1}=x^\omega$ is h-strong.
\end{proof}

Our next example is the usual basis of the pseudovariety \pv R of all
finite \Cl R-trivial monoids.

\begin{Thm}
  \label{t:R}
  For $\pv U=\pv M$, the pseudoidentity $(xy)^\omega x=(xy)^\omega$
  is h-strong.
\end{Thm}

\begin{proof}
  The $\omega$-reducibility of~\pv R for the equation $x=y$ was first
  proved in~\cite{Almeida&Costa&Zeitoun:2004}. In fact, the same holds
  for arbitrary systems of $\omega$-equations
  \cite{Almeida&Costa&Zeitoun:2005b}. The following basis of
  identities for the variety $\pv R^\omega$ was obtained
  in~\cite{Almeida&Zeitoun:2003b}:
  \begin{align*}
    &(x^r)^\omega=(x^\omega)^\omega=x^\omega\ (r\ge2)\\
    &(xy)z=x(yz),\ x1=1x=x \\
    &(xy)^\omega
      =(xy)^\omega x
      =(xy)^\omega x^\omega=x(yx)^\omega.
  \end{align*}
  Of all the above identities the only one that is not obviously
  provable from $(xy)^\omega x=(xy)^\omega$ is %
  $(xy)^\omega x^\omega=(xy)^\omega$. Yet, iterating the hypothesis,
  one gets $(xy)^\omega x^{n!}=(xy)^\omega$, whence also $(xy)^\omega
  x^\omega=(xy)^\omega$ by taking limits.
\end{proof}

While the two previous examples could be reformulated in the language
of semigroups, as the pseudovariety of semigroups generated in each
case has no additional monoids, for the next one the monoids in the
pseudovariety constitute a much smaller class. Indeed, we now consider
the pseudovariety \pv{LSl} of all finite semigroups which are locally
semilattices, that is, the pseudovariety defined by the set in the
following result.

\begin{Thm}
  \label{t:LSl}
  For $\pv U=\pv S$, the set of pseudoidentities
  $$\Sigma=\{x^\omega yx^\omega zx^\omega=x^\omega zx^\omega yx^\omega,
  x^\omega yx^\omega yx^\omega=x^\omega yx^\omega\}$$
  is h-strong.
\end{Thm}

\begin{proof}
  The $\omega$-reducibility of~\pv{LSl} for the equation $x=y$ was
  first proved in~\cite{Costa&Teixeira:2004}, where graph systems of
  equations are also considered.\footnote{Later, reducibility
    of~\pv{LSl} was extended to arbitrary systems of
    $\kappa$-equations \cite{Costa&Nogueira:2009}. In view of the
    well-known decomposition $\pv{LSl}=\pv{Sl}*\pv D$ (see, for
    instance, \cite[Section~10.8]{Almeida:1994a}), a generalization in
    a different direction has been obtained
    in~\cite{Costa&Nogueira&Teixeira:2016} where, in particular, it is
    proved that if \pv V is $\kappa$-reducible for the equation $x=y$,
    then so is $\pv V*\pv D$. This was later extended to the case of
    graph systems of equations
    in~\cite{Costa&Nogueira&Teixeira:2017}.} The following basis of
  identities for the variety \pv{LSl} may be found
  in~\cite{Costa&Nogueira:2008}:
  \begin{align*}
    &(x^r)^\omega
      =x^\omega x^\omega
      =x^\omega x
      =x^\omega \ (r\ge2)\\
    &(xy)z=x(yz),\
      (xy)^\omega x=x(yx)^\omega \\
    &x^\omega yx^\omega zx^\omega=x^\omega zx^\omega yx^\omega,
      x^\omega yx^\omega yx^\omega=x^\omega yx^\omega \\
    &(xy^\omega z)^\omega=(xy^\omega z)^2.
  \end{align*}
  Of all the above identities, the only one that requires a proof
  from~$\Sigma$ is the last one. We first note that the second
  pseudoidentity from~$\Sigma$ immediately infers %
  $(xy^\omega z)^3=(xy^\omega z)^2$, which entails %
  $(xy^\omega z)^{n+1}=(xy^\omega z)^n$ for every $n\ge2$, whence also
  $(xy^\omega z)^{n!}=(xy^\omega z)^2$. Taking limits, we may further
  prove $(xy^\omega z)^\omega=(xy^\omega z)^2$. In view of
  Proposition~\ref{p:general}, this concludes the proof of the
  theorem.
\end{proof}

\subsection{Monoids in which regular elements are idempotents}
\label{sec:DA}

The next example is given by two simple bases of pseudoidentities for
the pseudovariety \pv{DA} of all finite semigroups whose regular \Cl
D-classes are aperiodic subsemigroups, a property which is equivalent
to all regular elements being idempotents.

\begin{Thm}
  \label{t:DA}
  For $\pv U=\pv M$, the sets of pseudoidentities
  \begin{align*}
    \Sigma&=\{x^{\omega+1}=x^\omega,
              (xy)^\omega(yx)^\omega(xy)^\omega=(xy)^\omega\}\\
    \Gamma&=\{((xy)^\omega x)^2=(xy)^\omega x\}
  \end{align*}
  are h-strong.
\end{Thm}

\begin{proof}
  The $\omega$-reducibility of~\pv{DA} for the equation~$x=y$ has been
  proved in~\cite{Almeida&Costa&Zeitoun:2015b}. The following basis of
  identities for~$\pv{DA}^\omega$ has been recently
  obtained~\cite{Almeida&Kufleitner:2017}:
  \begin{align*}
    &(x^r)^\omega=(x^\omega)^\omega=x^\omega\ (r\ge2); \\
    &(xy)^\omega x=x(yx)^\omega; \\
    &(xy)z=x(yz),\ x1=1x=x \\
    &u^\omega vu^\omega=u^\omega,
  \end{align*}
  where $u,v\in\Om BU$ for an arbitrary finite alphabet $B$ and every
  variable that occurs in $v$ also occurs in~$u$.
  Only the family of identities in the last line requires a proof
  since the others are valid in every finite monoid.

  We start by showing that $\Sigma$ proves~$\Gamma$. We may prove
  algebraically:
  \begin{align*}
    (xy)^\omega %
    &=(xy)^\omega(yx)^\omega(xy)^\omega %
      =(xy)^\omega(yx)^{\omega+1}(xy)^\omega\\
    &=(xy)^\omega y\cdot(xy)^\omega\cdot x(xy)^\omega %
      =\cdots %
      =((xy)^\omega y)^{n!}(xy)^\omega (x(xy)^\omega)^{n!}.
  \end{align*}
  Taking limits, we get
  \begin{equation}
    (xy)^\omega
    =((xy)^\omega y)^\omega(xy)^\omega(x(xy)^\omega)^\omega
    \label{eq:DA-1}
  \end{equation}
  which, upon multiplication of both sides on the right by
  $x(xy)^\omega x$, yields
  \begin{align*}
    ((xy)^\omega x)^2 %
    &=((xy)^\omega y)^\omega(xy)^\omega(x(xy)^\omega)^{\omega+1}x \\
    &=((xy)^\omega y)^\omega(xy)^\omega(x(xy)^\omega)^\omega x
      \stackrel[\eqref{eq:DA-1}]{}=(xy)^\omega x.
  \end{align*}
  
  Conversely, we may algebraically prove $\Sigma$ from $\Gamma$ as
  follows. First, substituting $x$ for $y$ in the pseudoidentity
  $((xy)^\omega x)^2=(xy)^\omega x$ and multiplying both sides by
  $x^{\omega-1}$, we obtain
  $$x^{\omega+1}
  =x^{\omega-1}((xx)^\omega x)^2 %
  =x^{\omega-1}(xx)^\omega x %
  =x^\omega.$$
  Also, multiplying both sides of the pseudoidentity from~$\Gamma$
  on the right by $y(xy)^{\omega-1}$ yields %
  \begin{equation}
    (xy)^\omega x(xy)^\omega=(xy)^\omega.
    \label{eq:DA-2}
  \end{equation}
  Next, we obtain
  \begin{align*}
    (xy)^\omega
    &=(xy)^{\omega-1} xy(xy)^\omega
      =(xy)^{\omega-1} x(yx)^\omega y
      \stackrel[\eqref{eq:DA-2}]{}=(xy)^{\omega-1} x(yx)^\omega y(yx)^\omega y \\
    &=(xy)^\omega (yx)^\omega y
      =(xy)^\omega y(xy)^\omega,
  \end{align*}
  and so
  \begin{equation}
    (xy)^\omega y(xy)^\omega=(xy)^\omega,
    \label{eq:DA-3}
  \end{equation}
  which entails
  \begin{align*}
    (xy)^\omega
    &\stackrel[\eqref{eq:DA-3}]{}=(xy)^\omega y(xy)^\omega
      \stackrel[\eqref{eq:DA-2}]{}=(xy)^\omega y(xy)^\omega x(xy)^\omega \\
    &=(xy)^\omega (yx)^{\omega+1}(xy)^\omega
      =(xy)^\omega (yx)^\omega(xy)^\omega.
  \end{align*}

  From hereon, we are thus allowed to use $\Sigma\cup\Gamma$ as
  hypothesis in our proofs and we assume them without further mention.
  Recall that our objective is to prove that every $\omega$-identity
  of the last form $u^\omega vu^\omega=u^\omega$ under the above
  assumptions on $u$ and~$v$. Since $u$ and $v$ may expressed as limits of
  sequences of words with constant content, we may as well assume that
  both $u$ and $v$ are words.

  To simplify the notation we let $e=u^\omega$. We claim that the
  following statements hold:
  \begin{enumerate}[(i)]
  \item\label{item:DA-1} if $ewe=e$ may be proved, then $ew^2e=e$ may
    also be proved;
  \item\label{item:DA-2} if $eve=e=ewe$ may be proved, then $evwe=e$
    may also be proved;
  \item\label{item:DA-3} the pseudoidentity $(xyz)^\omega
    y(xyz)^\omega=(xyz)^\omega$ is provable.
  \end{enumerate}
  To establish (\ref{item:DA-1}), assume that we have proved $ewe=e$.
  Then we may also prove $ew=ewew=\cdots=(ew)^n$ and, taking limits,
  $ew=(ew)^\omega=(ew)^{\omega-1}$. Similarly, we can prove
  $we=(we)^\omega$. This yields the following equalities:
  $$e
  =ewe
  =(ew)^\omega e
  =(ew)^\omega(we)^\omega(ew)^\omega e
  =ew we ew e
  =ew^2 e.$$
  For~(\ref{item:DA-2}), suppose we have proved $eve=e=ewe$. Then, in
  view of~(\ref{item:DA-1}), we may prove
  $$e
  =ee =eweeve =e\cdot wev\cdot e =e\cdot wev\cdot wev\cdot e =evwe.$$
  To prove (\ref{item:DA-3}), observe first that, from~\eqref{eq:DA-2}
  and~\eqref{eq:DA-3} we deduce that the pseudoidentities
  $$(yzx)^\omega
  =(yzx)^\omega yz(yzx)^\omega
  =(yzx)^\omega y(yzx)^\omega
  =(yzx)^\omega x(yzx)^\omega$$
  are provable, so that, by~(\ref{item:DA-2}), so are
  $(yzx)^\omega=(yzx)^\omega yx(yzx)^\omega$ and
  \begin{align*}
    (xyz)^\omega
    &=(xyz)^{\omega+1}
      =x(yzx)^\omega yz
      =x(yzx)^\omega yzyx(yzx)^\omega yz \\
    &=(xyz)^{\omega+1}y(xyz)^{\omega+1}
      =(xyz)^\omega y(xyz)^\omega.
  \end{align*}
  Let $v=v_1\cdots v_n$, where the $v_i$ are letters. By
  (\ref{item:DA-3}) since every letter appearing in~$v$ also appears
  in~$u$, we may prove $ev_ie=e$ ($i=1,\ldots,n$). Applying
  (\ref{item:DA-2}) $n-1$ times, we deduce that we may also prove
  $eve=e$, as required.
\end{proof}

\subsection{\texorpdfstring{\Cl J}{J}-trivial monoids}
\label{sec:J}

The next example consists of the two bases commonly used to describe
the pseudovariety \pv J of all finite \Cl J-trivial semigroups.

\begin{Thm}
  \label{t:J}
  For $\pv U=\pv M$, the sets of pseudoidentities
  $$\Sigma=\{x^{\omega+1}=x^\omega, (xy)^\omega=(yx)^\omega\}
  \text{ and }%
  \Gamma=\{(xy)^\omega x=(xy)^\omega=y(xy)^\omega\}$$
  are h-strong.
\end{Thm}

\begin{proof}
  It is not difficult to show that \pv J is $\omega$-reducible for the
  equation $x=y$. In fact, \pv J is $\omega$-reducible for all finite
  systems of equations, and even of $\kappa$-equations
  \cite[Theorem~12.3]{Almeida:2002a}. On the other hand, the following
  basis of identities for the variety $\pv J^\omega$ is given
  in~\cite[Section~8.2]{Almeida:1994a}:
  \begin{align*}
    &(xy)z=x(yz),\
      x1=1x=x,\
      (x^\omega)^\omega=x^\omega\\
    &x^\omega x=xx^\omega=x^\omega\\
    &(xy)^\omega=(yx)^\omega=(x^\omega y^\omega)^\omega.
  \end{align*}
  The identities in the first line are valid in all finite monoids and,
  therefore require no proof. In view of Proposition~\ref{p:general},
  to finish the proof it suffices to show that the remaining
  identities are provable from both $\Sigma$ and $\Gamma$.

  In the case of~$\Sigma$, only the identity $(xy)^\omega=(x^\omega
  y^\omega)^\omega$ needs to be considered. The following describes a
  proof from~$\Sigma$. First, we do an algebraic proof:
  \begin{equation}
    \label{eq:J-step-a}
    (xy)^\omega
    =(xy)^{\omega+1}
    =x(yx)^\omega y
    =x(xy)^\omega y
    =\cdots=x^m(xy)^\omega y^m.
  \end{equation}
  Hence, we may also prove from $\Sigma$
  \begin{equation*}
    \label{eq:J-step-b}
    (xy)^\omega=(yx)^\omega=y^m(yx)^\omega x^m=y^m(xy)^\omega x^m,
  \end{equation*}
  which, combined with \eqref{eq:J-step-a}, yields
  \begin{equation}
    \label{eq:J-step-c}
    (xy)^\omega=x^my^m(xy)^\omega x^my^m
  \end{equation}
  Iterating \eqref{eq:J-step-c}, we get an algebraic proof of
  $(xy)^\omega=(x^my^m)^m(xy)^\omega (x^my^m)^m$. Letting $m=n!$ and
  taking limits, we obtain that $\Sigma$ proves
  \begin{equation}
    \label{eq:J-step1}
    (xy)^\omega
    =(x^\omega y^\omega)^\omega(xy)^\omega (x^\omega y^\omega)^\omega.
  \end{equation}
  Similarly, we may prove algebraically
  $$(x^\omega y^\omega)^\omega
  =x^\omega y^\omega(x^\omega y^\omega)^\omega
  =x x^\omega y^\omega(x^\omega y^\omega)^\omega
  =x (x^\omega y^\omega)^\omega$$
  and so also
  $$(x^\omega y^\omega)^\omega %
  =xy(x^\omega y^\omega)^\omega %
  =\cdots %
  =(xy)^{n!}(x^\omega y^\omega)^\omega.
  $$
  Taking limits, we get %
  $(x^\omega y^\omega)^\omega %
  =(xy)^\omega(x^\omega y^\omega)^\omega$. %
  Combining with~\eqref{eq:J-step1} and taking into account that
  $(x^\omega y^\omega)^\omega$ is idempotent, we finally complete the
  proof of $(xy)^\omega=(x^\omega y^\omega)^\omega$ from~$\Sigma$.

  For $\Gamma$, we first note that, substituting $x$ for $y$
  in~$(xy)^\omega x=(xy)^\omega$, yields $x^{\omega+1}=x^\omega$.
  Hence, it suffices to show that the pseudoidentity
  $(xy)^\omega=(yx)^\omega$ is provable from~$\Gamma$, which can be
  established algebraically:
  $$(xy)^\omega %
  =y(xy)^\omega %
  =(yx)^\omega y %
  =(yx)^\omega.\qedhere$$
\end{proof}

Note that, in the proof of Theorem~\ref{t:J}, we alternated several
times topological and transitive closure. More precisely, we actually
proved that $\tilde{\Sigma}=\Sigma_4$. We do not know whether
$\tilde{\Sigma}=\Sigma_3$ but show below that
$\tilde{\Sigma}\ne\Sigma_2$.

We start with an auxiliary lemma involving equidivisibility. We say
that a semigroup $S$ is \emph{equidivisible} if any two factorizations
of the same element admit a common refinement
\cite{McKnight&Storey:1969}. We say that a pseudovariety (of
semigroups or monoids) \pv V is \emph{equidivisible} if, for each
finite set $A$, the semigroup \Om AV is equidivisible. The
equidivisible pseudovarieties of semigroups have been characterized
in~\cite{Almeida&ACosta:2017}. The characterization of equidivisible
pseudovarieties of monoids can be derived from it by noting that, for
a pseudovariety \pv V of monoids, $\Om AV=(\Om AW)^1$, where \pv W is
the pseudovariety of semigroups generated by~\pv V, which amounts to a
simple exercise, together with the obvious observation that a
semigroup $S$ is equidivisible if and only if so is the monoid $S^1$.
In particular, \pv M is equidivisible.

The following lemma can surely be generalized but is already
sufficient for our purposes.

\begin{Lemma}
  \label{l:omega-as-factor}
  Let $w\in A^*$ and $u,v\in\Om AM$ be such that $w^\omega$ is a
  factor of~$uv$. Then, $w^\omega$ is a factor of at least one of the
  factors $u$ and~$v$.
\end{Lemma}

\begin{proof}
  By equidivisibility, from the two factorizations $uv=xw^\omega y$
  (for some $x,y\in\Om AM$), we know that there is a common
  refinement. Hence, if $w^\omega$ is a factor of neither $u$ nor~$v$,
  then there is a factorization $w^\omega=zt$ with $u=xz$ and $v=ty$.
  We reach a contradiction by showing that $w^\omega$ must be a factor
  of at least one of $z$ and~$t$.

  First note that at least one of $z$ and $t$ is a not a finite word
  for, otherwise, so would be $w^\omega$. By symmetry, we may as well
  assume that $z$~is not a finite word. We claim that $w^n$ is a
  prefix of~$z$ for every $n\ge1$ and, therefore, so is $w^\omega$,
  thereby reaching the desired contradiction. To prove the claim,
  consider the monoid $M_k$ consisting of all words of~$A^*$ of length
  at most $k=|w^n|$ where the product is defined by $r\cdot s=rs$ if
  $|rs|\le k$, while $r\cdot s$ is taken to be the prefix of $rs$ of
  length $k$ otherwise. Consider also the unique continuous
  homomorphism $\varphi_k:\Om AM\to M_k$ which maps each letter $a$
  from~$A$ to $a$ as an element of~$M_k$. Note that $\varphi_k$ maps
  each word of length at most $k$ to itself and every other finite
  word to its prefix of length $k$. It follows that $\varphi_k(s)$ is
  a prefix of~$s$ for every pseudoword $s\in\Om AM$.

  Since $z$ is not a finite word, there is a sequence of words
  $(z_m)_m$ converging to~$z$ with $|z_m|\ge k$ for every~$m$. Then,
  from the equalities
  $\varphi_k(z)=\varphi_k(zt)=\varphi_k(w^\omega)=w^n$, we deduce that
  $w^n$ is a prefix of~$z$, as was claimed.
\end{proof}

\begin{Prop}
  \label{p:J-notSigma2}
  For $\pv U=\pv M$ and $\Sigma=\{x^{\omega+1}=x^\omega,\
  (xy)^\omega=(yx)^\omega\}$, we have $\tilde{\Sigma}\ne\Sigma_2$.
\end{Prop}

\begin{proof}
  For the purpose of the present proof, we take $A=\{x,y\}$.
  
  We have shown in the proof of Theorem~\ref{t:J} that the
  $\omega$-identity
  \begin{equation}
    \label{eq:J-noSigma2}
    (x^\omega y^\omega)^\omega=(xy)^\omega
  \end{equation}
  belongs to~$\Sigma_3$. We prove that it does belong to~$\Sigma_2$.
  For that purpose, we claim that, if the pseudoidentity
  $w=(xy)^\omega$ is in~$\Sigma_2$, then $(xy)^\omega$ is a factor
  of~$w$. Since not even the word $xyx$ is a factor of~$(x^\omega
  y^\omega)^\omega$ (cf.~\cite[Lemma~8.2]{Almeida&Volkov:2006}), we
  conclude that the pseudoidentity \eqref{eq:J-noSigma2} cannot belong
  to~$\Sigma_2$.

  The hypothesis of the claim implies the existence of a sequence of
  pseudoidentities $(w_n=v_n)_n$ in~$\Sigma_1$ which converges to
  $w=(xy)^\omega$. Now, by taking subsequences, we may as well assume
  that either the sequence $(v_n)_n$ consists only of finite words or
  only of infinite pseudowords. In the first case, the only
  $\Sigma_1$-pseudoidentity of the form $u=v_n$ is the trivial
  pseudoidentity $v_n=v_n$. Since the pseudoidentity
  \eqref{eq:J-noSigma2}~is not trivial over~\pv M, the first case is
  excluded. In the second case, we note that $(xy)^\omega$ is a factor
  of~$v_n$ for every sufficiently large~$n$. Indeed, the language
  $(xy)^*$ has an open closure in~\Om AM
  \cite[Theorem~3.6.1]{Almeida:1994a}, which consists of all powers
  of~$xy$.

  Thus, to establish our claim, it suffices to show that, if the
  pseudoidentity $u=v$ belongs to~$\Sigma_1$ and the pseudoword
  $(xy)^\omega$ is a factor of~$v$ then it is also a factor of~$u$.
  Since $\Sigma_1$~is the transitive closure of~$\Sigma_0$, by a
  straightforward induction argument it suffices to treat the case
  where $u=v$ belongs to $\Sigma_0$. Hence, there are a nonempty word
  $\mathbf{t}$, pseudowords $w_1,\ldots,w_n$, a pseudoidentity $u'=v'$
  such that either it or $v'=u'$ belongs to~$\Sigma$, and a continuous
  endomorphism $\varphi$ of~\Om AM such that
  $u=\mathbf{t}(\varphi(u'),w_1,\ldots,w_n)$ and
  $v=\mathbf{t}(\varphi(v'),w_1,\ldots,w_n)$. Note that, since $u'$
  and $v'$ are \Cl J-equivalent, so are $\varphi(u')$ and
  $\varphi(v')$, which means that these two pseudowords have the same
  factors.
  Without loss of generality, we may assume that, writing
  $\mathbf{t}=\mathbf{t}(x_0,x_1,\ldots,x_n)$, all the letters $x_i$
  ($i=1,\ldots,n$) appear at least once in~$\mathbf{t}$. By
  Lemma~\ref{l:omega-as-factor}, we deduce from the assumption that
  $(xy)^\omega$~is a factor of~$v$ that it must also be a factor of
  either $\varphi(v')$ or one of the $w_i$. Hence, $(xy)^\omega$~is a
  factor of either $\varphi(u')$ or one of the $w_i$ and, therefore,
  also of~$u$.
\end{proof}

\section{The group case}
\label{sec:groups}

Let \pv G be the pseudovariety of all finite groups. As far as
subpseudovarieties of~\pv G are concerned, whether we view groups in
the natural signatures for semigroups, monoids, or groups is
irrelevant, since the identity element is the only idempotent
$x^\omega$ and inversion is also captured by the semigroup pseudoword
$x^{\omega-1}$. However, for the purpose of this section, we prefer to
deal with the group signature, consisting of a binary multiplication,
a constant symbol 1, for the identity element, and the unary operation
of inversion.

Note that for $u,v\in\Om AG$, each of the pseudoidentities $u=v$ and
$u^{-1}v=1$ is provable from the other. Hence, in the language of
groups, it suffices to deal with pseudoidentities of the form $w=1$.

For a set $\Upsilon$ of group pseudowords, consider the pseudovariety
$$\pv H_\Upsilon=\op u=1:\ u\in\Upsilon\cl.$$
Denote by $N_\Upsilon$ the subgroup of~\Om AG generated by all
conjugates of elements of the form $\varphi(u)$ with $B$ a finite
alphabet, $u\in\Upsilon\cap\Om BG$, and $\varphi:\Om BG\to\Om AG$ a
continuous homomorphism.

\begin{Thm}
  \label{t:conjecture-G}
  Let $\Upsilon$ be a set of group pseudowords. Then, for every
  $w\in\Om AG$, the pseudovariety $\pv H_\Upsilon$ satisfies the
  pseudoidentity $w=1$ if and only if $w$ belongs to the closure of
  $N_\Upsilon$ in~\Om AG.
\end{Thm}

\begin{proof}
  Suppose that $w\in\overline{N_\Upsilon}$, say $w=\lim w_n$ for a sequence
  $(w_n)_n$ of elements of~$N_\Upsilon$. Now, for each $u\in\Upsilon\cap\Om BG$
  and each continuous homomorphism $\varphi:\Om BG\to\Om AG$, the
  pseudoidentity $\varphi(u)=1$ is valid in~$\pv H_\Upsilon$, whence so
  is $g\varphi(u)g^{-1}=1$ for every $g\in\Om AG$. Hence, each
  pseudoidentity $w_n=1$ is also valid in~$\pv H_\Upsilon$, which
  entails that so is $w=1$.

  Conversely, suppose that $w\in\Om AG$ is such that $\pv H_\Upsilon$
  satisfies the pseudoidentity $w=1$ but, arguing by contradiction,
  $w$ does not belong to~$\overline{N_\Upsilon}$. Since \Om AG is a
  profinite group, for each $v\in\overline{N_\Upsilon}$ there is a
  clopen normal subgroup $K_v$ of~\Om AG such that $wK_v\ne vK_v$. As
  $\overline{N_\Upsilon}$ is compact, we may extract a finite covering
  from $\bigcup_{v\in\overline{N_\Upsilon}}vK_v$, say
  $\bigcup_{i=1}^nv_iK_{v_i}$. Let $K=\bigcap_{i=1}^nK_{v_i}$, which
  is a clopen normal subgroup of~\Om AG. Finally, let
  $L=\overline{N_\Upsilon}K$, which is again a clopen normal subgroup
  of~\Om AG. Clearly, $L$ contains $\overline{N_\Upsilon}$ and we
  claim that it does not contain~$w$. Indeed, if $w\in L$, then there
  is $v\in\overline{N_\Upsilon}$ such that $w\in vK$. Since the open
  sets $v_iK_{v_i}$ cover $\overline{N_\Upsilon}$, we deduce that
  there is $i\in\{1,\ldots,n\}$ such that $v\in v_iK_{v_i}$. Hence, we
  have
  $$w\in vK\subseteq v_iK_{v_i}K\subseteq v_iK_{v_i},$$
  which contradicts the choice of $K_{v_i}$.

  Let $G$ be the quotient group $\Om AG/L$, which is a finite group.
  Since $L$ contains~$\overline{N_\Upsilon}$, by the first part of the
  proof we infer that $G\in\pv H_\Upsilon$. Hence, by hypothesis, $G$
  satisfies the pseudoidentity $w=1$ which, for the evaluation of each
  letter $a\in A$ by the coset $aL$, yields $w\in L$, in contradiction
  with the choice of~$L$. It follows that $w$ must belong
  to~$\overline{N_\Upsilon}$.
\end{proof}

\begin{Cor}
  \label{c:conjecture-group-case}
  The pseudovariety \pv G is strong.
\end{Cor}

The same method may be applied to show that the pseudovariety of all
finite rings (not necessarily with identity element) is strong.

\begin{Thm}
  \label{t:G}
  For $\pv U=\pv M$, the pseudoidentity $x^\omega=1$ is h-strong.
\end{Thm}

\begin{proof}
  The key result upon which the proof is based is the
  $\kappa$-reducibility of $\pv G=\op x^\omega=1\cl$ for the equation
  $x=y$. More generally, based on a celebrated theorem of
  Ash~\cite{Ash:1991}, which contains all the essential hard work, it
  has been observed in~\cite[Theorem~4.9]{Almeida&Steinberg:2000a}
  that \pv G is $\kappa$-reducible for all finite graph systems of
  equations.\footnote{By a graph system of equations we mean the
    system associated with a finite directed graph which, for each
    edge $x\xrightarrow{y}z$ in the graph, includes the equation
    $xy=z$. Now, for~\pv G, every solution of the equation $y_1=y_2$
    is a solution of the graph system determined by the graph %
    $
    \begin{tikzpicture}[baseline=-\the\dimexpr\fontdimen22\textfont2\relax,
      semithick]
      \node   (1)                {$x$};
      \node   (2) [right=of 1]   {$z$};
      \path[->]  (1) edge [bend left=14]  node [above] {\tiny$y_1$} (2)
                 (1) edge [bend right=14] node [below] {\tiny$y_2$} (2);
    \end{tikzpicture}
    $ %
    and, conversely, every solution of the graph system involves a
    solution of the equation $y_1=y_2$.} In view of
  Proposition~\ref{p:general}, it remains to show that every
  $\kappa$-identity valid in~\pv G may be proved from $x^\omega=1$ in
  the sense of Section~\ref{sec:scheme}. In fact, we show that it may
  be algebraically proved.

  The essential observation is that the $\kappa$-identity
  \begin{equation}
    \label{eq:inversion-inverts-order}
    (xy)^{\omega-1}=y^{\omega-1}x^{\omega-1}
  \end{equation}
  may be algebraically
  proved from $x^\omega=1$. Indeed here is a description of such a
  proof:
  \begin{align*}
    (xy)^{\omega-1}
    &=y^\omega(xy)^{\omega-1}x^\omega
    =y^{\omega-1}y(xy)^{\omega-1}xx^{\omega-1}\\
    &=y^{\omega-1}(yx)^\omega x^{\omega-1}
    =y^{\omega-1}x^{\omega-1}.
  \end{align*}
  Note also that $(x^{\omega-1})^{\omega-1}=x^{\omega+1}=x$, where the
  latter equality is provable from~$x^\omega=1$. Applying repeatedly
  the identity~\eqref{eq:inversion-inverts-order}, every $\kappa$-term
  may be rewritten in the form $x_1^{\varepsilon_1}\cdots
  x_n^{\varepsilon_n}$, where each $x_i$~is a variable and each
  exponent $\varepsilon_i$~is either $1$ or~$\omega-1$. We say that a
  $\kappa$-term of this type is in \emph{standard form}; a
  $\kappa$-identity whose sides are in standard form is also said to
  be in \emph{standard form}. Thus, in the presence of the
  pseudoidentity $x^\omega=1$, every $\kappa$-identity valid in~\pv G
  is algebraically provably equivalent to a $\kappa$-identity $u=v$ in
  standard form that is also valid in~$\pv G$. Since the free group is
  residually finite, $u=v$ may be viewed as a group identity (by
  removing the $\omega$'s from the exponents) that is satisfied by all
  groups, which means that, by applying the reduction rules
  $aa^{-1}\to1$ and $a^{-1}a\to1$ a finite number of times, both sides
  may be transformed to the same reduced group word. Since, in the
  context of finite monoids, both rules follow from the pseudoidentity
  $x^\omega=1$, we conclude that $u=v$ is algebraically provable from
  $x^\omega=1$.
\end{proof}

Although the next result is superseded by Corollary~\ref{c:CS}, it
seems worthwhile to include it at this stage.

\begin{Cor}
  \label{c:H}
  For $\pv U=\pv M$, every set $\Sigma$ of pseudoidentities defining a
  group pseudovariety is h-strong.
\end{Cor}

\begin{proof}
  It suffices to apply Proposition~\ref{p:transfer-up} taking into
  account Theorem~\ref{t:G}, Corollary~\ref{c:conjecture-group-case},
  and Proposition~\ref{p:t-strong}(\ref{item:G-t-strong}).
\end{proof}

\section{The completely simple semigroup case}
\label{sec:CS}

Let \pv{CS} be the pseudovariety of all finite completely simple
semigroups. It is defined, for instance, by the pseudoidentity
$(xy)^\omega x=x$. By a well-known theorem of Rees, completely simple
semigroups are precisely those that admit a Rees matrix representation
$\Cl M(I,G,\Lambda,P)$, where $I$ and $\Lambda$ are sets, $G$~is a
group, and $P:\Lambda\times I\to G$ is a function (the \emph{sandwich
  matrix}, the image $P(\lambda,i)$ being usually denoted
$p_{\lambda,i}$); as a set, it is the Cartesian product $I\times
G\times \Lambda$, and multiplication is given by the formula
$$(i,g,\lambda)(j,h,\mu)=(i,gp_{\lambda,j}h,\mu).$$
Assuming that $1$ is a common element of~$I$ and~$\Lambda$, the sandwich
matrix may be supposed to be \emph{normalized} in the sense that
$p_{1,i}=p_{\lambda,1}=1$ for all $i\in I$ and $\lambda\in\Lambda$.

The purpose of this section is to establish that \pv{CS} is strong.
The key ingredient is the following characterization of the
congruences on a Rees matrix semigroup which may be extracted
from~\cite[Theorem~10.48]{Clifford&Preston:1967}.

\begin{Thm}
  \label{t:CS-congruences}
  Let $S=\Cl M(I,G,\Lambda,P)$ be a Rees matrix semigroup and let $\rho$
  be a congruence on~$S$. Consider the relations
  \begin{align*}
    \rho_1&=\{(i,j)\in I\times I: (i,1,1)\mathrel{\rho}(j,1,1)\} \\
    \rho_2&=\{(\lambda,\mu)\in \Lambda\times\Lambda:
                  (1,1,\lambda)\mathrel{\rho}(1,1,\mu)\} \\
    N_\rho&=\{g\in G: (1,g,1)\mathrel{\rho}(1,1,1)\}.
  \end{align*}
  Then $\rho_1$ (respectively $\rho_2$) is an equivalence relation on
  the set $I$ (resp.~$\Lambda$) and $N$ is a normal subgroup of~$G$
  such that
  \begin{align}
    \label{eq:CS-cong-1}
    (i,j)\in\rho_1&\implies p_{\lambda,i}N_\rho=p_{\lambda,j}N_\rho \\
    \label{eq:CS-cong-2}
    (\lambda,\mu)\in\rho_2&\implies p_{\lambda,i}N_\rho=p_{\mu,i}N_\rho.
  \end{align}
  Conversely, for every triple $\tau=(\rho_1,\rho_2,N_\rho)$, where
  $\rho_1$ (resp.~$\rho_2$) is an equivalence relation on the set $I$
  (resp.~$\Lambda$) and $N_\rho$ is a normal subgroup of~$G$,
  satisfying properties \eqref{eq:CS-cong-1} and~\eqref{eq:CS-cong-2},
  the relation
  $$\rho_\tau=\left\{\bigl((i,g,\lambda),(j,h,\mu)\bigr)\in S\times S:
    i\mathrel{\rho_1}j,\ 
    \lambda\mathrel{\rho_2}\mu,\ 
    gN_\rho=hN_\rho\right\}$$
  is a congruence on~$S$ and every congruence on $S$ is of this form.
\end{Thm}

We may now proceed as in the proof of Theorem~\ref{t:conjecture-G} to
obtain our next result. For a set $\Sigma$ of
\pv{CS}-pseudoidentities, consider the closed congruence
$\rho=\tilde{\Sigma}$ generated by the pairs of the form
$(\varphi(u),\varphi(v))$ with $u=v$ in~$\Sigma$, $u,v\in\Om B{CS}$,
and $\varphi:\Om B{CS}\to\Om A{CS}$ a continuous homomorphism.

\begin{Thm}
  \label{t:conjecture-CS}
  The pseudovariety \pv{CS} is strong.
\end{Thm}

\begin{proof}
  Suppose that $u,v\in\Om A{CS}$ are such that the pseudoidentity
  $u=v$ is valid in~$\op\Sigma\cl$. We claim that $u=v$ is provable
  from~$\Sigma$, which amounts to the condition $u\mathrel{\rho}v$.

  Arguing by contradiction, suppose that $(u,v)\notin\rho$.
  By~\cite{Almeida:1991d}, there is an isomorphism $\psi:\Om A{CS}\to
  S=\Cl M(A,\Om XG,A,P)$ where, choosing an element $a_0$ from~$A$ and
  letting $A'=(A\setminus\{a_0\})^2$ be the Cartesian square, we have
  $X=A\cup A'$, $p_{a_0,b}=p_{b,a_0}=1$, and $p_{a,b}=(a,b)$ for each
  $(a,b)\in A'$. Note that the letter $a_0$ plays the role of $1$ in
  the normalization of the sandwich matrix.

  Consider the normal subgroup $N_\rho$ and the equivalence relations
  $\rho_1$ and $\rho_2$ as defined in Theorem~\ref{t:CS-congruences}.
  Note that $N_\rho$ is a closed normal subgroup of~\Om XG.

  Let $\psi(u)=(a,g,b)$ and $\psi(v)=(c,h,d)$. Since we are assuming
  that $(u,v)\notin\rho$, at least one of the following conditions
  must hold:
  $$(a,c)\notin\rho_1,\ %
  (b,d)\notin\rho_2,\ %
  gN_\rho\ne hN_\rho.$$ In case one of the first two conditions holds,
  the mapping from %
  $S$ onto the rectangular band $T=A/\rho_1\times A/\rho_2$ that maps
  each triple $(x,w,y)\in S$ to~$(x/\rho_1,y/\rho_2)$ is a continuous
  homomorphism onto a semigroup from~\pv{CS} that distinguishes $u$
  and $v$. Moreover, its kernel congruence is contained in~$\rho$,
  which implies that $T\in\op\Sigma\cl$, contradicting the assumption
  that $\op\Sigma\cl$ satisfies $u=v$. Hence, we may assume that
  $(a,c)\in\rho_1$ and $(b,d)\in\rho_2$, so that $g^{\omega-1}h$ does
  not belong to~$N_\rho$. As in the proof of
  Theorem~\ref{t:conjecture-G}, we deduce that there is a clopen
  normal subgroup $K$ of~\Om XG such that $N_\rho\subseteq K$ and
  $g^{\omega-1}h\notin K$. Since $K$ contains $N_\rho$, the triple
  $(\rho_1,\rho_2,K)$ still satisfies the analogues of
  conditions~\eqref{eq:CS-cong-1} and~\eqref{eq:CS-cong-2}. Hence, by
  Theorem~\ref{t:CS-congruences}, it defines a congruence $\bar\rho$
  on~$S$. Since $K$~has finite index in~\Om XG, the congruence
  $\bar\rho$ has finite index in~$S$. The reader may easily verify
  that, since $\rho$ is a closed congruence on~$S$ and $K$ is a closed
  subgroup of~\Om XG, the congruence $\bar\rho$ is still closed,
  whence the natural mapping $S\to S/\bar\rho$ is a continuous
  homomorphism. As $\bar\rho$ contains~$\rho$, the quotient semigroup
  $S/\bar\rho$ satisfies $\Sigma$ but, by construction, fails the
  pseudoidentity $u=v$. This completes the proof.
\end{proof}

We proceed with some observations on the variety \Cl{CS} of completely
simple semigroups, consisting of algebras with a binary multiplication
and unary ``inversion'' $\_^{-1}$ satisfying the following identities,
where $u^0$ abbreviates $uu^{-1}$:
\begin{align}
  &(xy)z=x(yz),\ x^{-1}x=x^0,\ x^0x=x,\ (x^{-1})^{-1}=x
    \label{eq:CR-eqs}\\
  &(xyx)^0=x^0.
    \label{eq:CS-eqs}
\end{align}
The identities~\eqref{eq:CR-eqs} define the variety of completely
regular semigroups \cite{Petrich&Reilly:1999}. In the presence of
them, it is well known that the identity~\eqref{eq:CS-eqs} is
equivalent to
\begin{equation}
  \label{eq:CS-eqs-alt}
  (xy)^0x=x.
\end{equation}
Note that, when the inversion operation $\_^{-1}$ is interpreted as
$\_^{\omega-1}$ in a finite semigroup, $\_^0$ becomes $\_^\omega$ and
the first two identities in~\eqref{eq:CR-eqs} are verified while the
last two identities in~\eqref{eq:CR-eqs} are valid in every completely
regular finite semigroup.

\begin{Thm}
  \label{t:CS}
  For $\pv U=\pv S$, each of the sets of pseudoidentities
  $$\Sigma=\{(xyx)^\omega=x^\omega,x^{\omega+1}=x\}
  \text{ and } %
  \Gamma=\{(xy)^\omega x=x\}$$
  is h-strong.
\end{Thm}

\begin{proof}
  As implied by the entry for \pv{CS}
  in~\cite[Table~2]{Almeida&Steinberg:2000a}, the methods
  of~\cite{Almeida:1997a} show that the pseudovariety \pv{CS} is
  $\kappa$-reducible for graph systems of equations. In fact,
  in~\cite{Almeida:1997a} vertices were allowed to be constrained by
  the clopen subset $\{1\}\subseteq(\Om AS)^1$, which forces the
  corresponding variable to be evaluated by~$1$. It follows that
  \pv{CS} is $\kappa$-reducible for the equation $x=y$
  
  Consider next the variety $\pv{CS}^\kappa$. Since free completely
  simple semigroups are residually finite
  \cite[Proposition~2.5]{Pastijn&Trotter:1988}, in view of the above
  remarks a basis of identities for~$\pv{CS}^\kappa$ is given by
  \begin{align*}
    &(xy)z=x(yz),\
      x^{\omega-1}x=xx^{\omega-1},\
      xx^{\omega-1}x=x,\
      (x^{\omega-1})^{\omega-1}=x \\
    &xyx(xyx)^{\omega-1}=xx^{\omega-1}.
  \end{align*}
  All the identities in the first line are obviously provable
  from~$\Sigma$. On the other hand, substituting $x$ for $y$ gives a
  proof of $x^{\omega+1}=x$ from~$\Gamma$ and we already observed that
  in the presence of the identities in the above first line, the
  identities $(xyx)^\omega=x^\omega$ and $(xy)^\omega x=x$ are provable
  from each other. Thus, the result follows from
  Proposition~\ref{p:general}.
\end{proof}

Combining Theorems~\ref{t:conjecture-CS} and \ref{t:CS} with
Proposition~\ref{p:t-strong}(\ref{item:CS-t-strong}), and applying
Proposition~\ref{p:transfer-up}, we obtain the following result.

\begin{Cor}
  \label{c:CS}
  For $\pv U=\pv S$, every set $\Sigma$ of pseudoidentities such that
  $\op\Sigma\cl\subseteq\pv{CS}$ is h-strong.\qed
\end{Cor}

Taking into account the $\kappa$-reducibility for the equation $x=y$
of the pseudovariety $\pv{CR}=\op x^{\omega+1}=x\cl$, of all finite
completely regular semigroups, (cf.~\cite{Almeida&Trotter:2001}) and
the residual finiteness of free completely regular semigroups (viewed
as algebras in the signature $\{\_.\_,\, \_^{-1}\}$), the following
theorem can be proved in the same way as Theorem~\ref{t:CS}.

\begin{Thm}
  \label{t:CR}
  For $\pv U=\pv S$, the pseudoidentity $x^{\omega+1}=x$ is h-strong.
\end{Thm}

\section{The commutative case}
\label{sec:Com}

Our next example is the usual basis of the pseudovariety \pv{Com} of
all finite commutative monoids.

\begin{Thm}
  \label{t:Com}
  For $\pv U=\pv M$, the identity $xy=yx$ is h-strong.
\end{Thm}

\begin{proof}
  Although \cite{Steinberg:1997a} deals with the notion of
  hyperdecidability, which is weaker than $\sigma$-reducibility
  provided the implicit signature $\sigma$ has suitable algorithmic
  properties
  (cf.~\cite{Almeida&Steinberg:2000a,Almeida&Steinberg:2000b}), it is
  observed in~\cite{Almeida&Steinberg:2000a} that the same methods
  show that \pv{Com} is $\kappa$-reducible, in particular for the
  equation $x=y$. We claim that the following is a basis of identities
  for the variety $\pv{Com}^\kappa$, as usual, we write $x^\omega$ for
  the product $x^{\omega-1}x$:
  \begin{align*}
      &x^{\omega-1}x^\omega=x^{\omega-1},\
        (x^{\omega-1})^{\omega-1}=x^{\omega+1} \\
      &(xy)z=x(yz),\ x1=1x=x \\
      &xy=yx \\
      &(xy)^{\omega-1}=x^{\omega-1}y^{\omega-1}.
  \end{align*}
  Indeed all such identities are clearly valid in~\pv{Com}. On the
  other hand, using the above identities, one may reduce every
  $\kappa$-term to one of the form $u=x_1^{\varepsilon_1}\cdots
  x_n^{\varepsilon_n}$, where each exponent $\varepsilon_i$ is either
  $1$ or~$\omega-1$. Moreover, we may rearrange the factors so that
  the powers with the same base are not separated by other powers. On
  the other hand, powers with the same base can collected together to
  powers of one of the forms $x^n$ or $x^{\omega-1}x^n$, where $n$ is
  a nonnegative integer, or $(x^{\omega-1})^n$, where $n\ge2$. For
  convenience, we write $x^{\omega-1}x^n$ as $x^{\omega-1+n}$ and
  $(x^{\omega-1})^n$ as $x^{(\omega-1)n}$. In this form, we say that
  we have a $\kappa$-term in \emph{completely reduced form}. Consider
  two $\kappa$-terms in completely reduced form
  $u=a_1^{\varepsilon_1}\cdots a_n^{\varepsilon_n}$ and
  $v=a_1^{\delta_1}\cdots a_n^{\delta_n}$ over the alphabet
  $\{a_1,\ldots,a_n\}$ and assume that the identity $u=v$ is valid
  in~\pv{Com}. Then, substituting $1$ for every variable but one, we
  obtain an identity valid in~\pv{Com} of the form
  $x^\varepsilon=x^\delta$ with $\varepsilon$ and $\delta$ exponents
  of one of the forms $n$, $\omega-1+n$ or $(\omega-1)n$.
  By considering a suitable monogenic finite semigroup, one
  immediately verifies that $\varepsilon=\delta$. This proves the
  claim.

  To conclude the proof, it remains to observe that the last identity
  in the above basis is provable from $xy=yx$. Indeed, each identity
  $(xy)^{n!-1}=x^{n!-1}y^{n!-1}$ can be easily proved and the result
  follows by taking limits.
\end{proof}

We next show that the pseudoidentity $xy=yx$ also defines a strong
pseudovariety.

\begin{Thm}
  \label{t:strong-Com}
  The pseudovariety \pv{Com} is strong.
\end{Thm}

\begin{proof}
  First of all, we recall that elements of the free profinite monoid
  $\Om A{Com}$ over the set $A=\{a_1,\dots ,a_n\}$ can be written in
  the form $a_1^{\varepsilon_1}\cdots a_n^{\varepsilon_n}$, where
  $\varepsilon_1, \dots , \varepsilon_n\in \hat{\mathbb N}$. Each
  pseudoidentity $a_1^{\varepsilon_1}\cdots
  a_n^{\varepsilon_n}=a_1^{\delta_1}\cdots a_n^{\delta_n}$ is provably
  equivalent
  to the set of pseudoidentities
  $\{a_1^{\varepsilon_1}=a_1^{\delta_1}, \dots ,
  a_n^{\varepsilon_n}=a_n^{\delta_n}\}$. Therefore, we may assume that
  all pseudoidentities which are under consideration are over a single
  variable.

  Recall that $\Om 1{Com}$ and $\Om 1{M}$ are both isomorphic to
  $\hat{\mathbb{N}}$ via the mapping %
  $\varepsilon \mapsto x^\varepsilon$ and that
  $\hat{\mathbb{Z}}=\hat{\mathbb{N}}\setminus \mathbb{N}$ is
  isomorphic to $\Om 1G$.
  For the purpose of simplification of notation, we identify
  isomorphic structures, so, for example, we denote by $\pi$ the
  continuous homomorphisms $\pi : \Om 1{Com} \rightarrow \Om 1G$ given
  by the rule $\pi(x^\varepsilon)=x^{\omega+\varepsilon}$, which is
  formally a projection $\pi_{\pv G} : \Om 1{M} \rightarrow \Om 1G$
  after the identification $\Om 1{Com}= \Om 1{M}$. In this way, the
  group $\Om 1G$ can be viewed as a retract of the monoid $\Om
  1{Com}$.
  
  Let $\Sigma $ be a set of pseudoidentities in the variable $x$ and
  $x^\varepsilon=x^\delta$ be a nontrivial pseudoidentity satisfied
  by~$\op\Sigma\cl$. We need to show that $x^\varepsilon=x^\delta$
  belongs to $\tilde{\Sigma}$.
 
  By the \emph{finite index} $i(\Sigma)$ of $\Sigma$ is meant the
  minimal natural number $n$ such that $(x^n=x^\mu) \in \Sigma$ where
  $\mu\not= n$. If such $n\in \mathbb N$ does not exist, we say that
  $\Sigma$ has \emph{infinite index}.
 
  If $n$ is the index of~$\Sigma$, then in a nontrivial pseudoidentity
  $(x^n=x^\mu)\in\Sigma$ we have $\mu$ infinite or $n<\mu\in\mathbb
  N$. In both cases, in view of Lemma~\ref{l:1-var-pseudoidentities}
  the pseudoidentity $x^n=x^{\omega+n}$ is provable from~$\Sigma$.
  On the other hand, the monogenic monoid $C_{n,1}^1$
  belongs to $\op\Sigma\cl$, whence it satisfies
  $x^\varepsilon=x^\delta$. Hence, both $\varepsilon$ and $\delta$ are
  infinite or
  greater than or equal to $i(\Sigma)=n$. From
  Lemma~\ref{l:1-var-pseudoidentities}, it follows that
  $x^\varepsilon
  =x^{\omega+\varepsilon}$ and
  $x^\delta
  =x^{\omega+\delta}$ are
  provable from $\Sigma$. Therefore,
  $x^\varepsilon=x^\delta$ is provable from~$\Sigma$ if and only if so
  is $x^{\omega+\varepsilon}=x^{\omega+\delta}$. On the other hand, if
  $\Sigma$ has infinite index, then the monoid $C_{n,1}^1$ belongs to
  $\op\Sigma\cl$ for every $n\in\mathbb N$ and we see that
  $\varepsilon,\delta\not \in \mathbb N$. Altogether, we can deal just
  with the case $\varepsilon,\delta\in \hat{\mathbb Z}$.

  Before finishing the proof, we make one technical observation. Let
  $\rho$ be the restriction of the relation $\tilde{\Sigma}$ to the
  set $\{x^\lambda : \lambda\in \hat{\mathbb Z}\}$ and $\tau$ denote
  the relation $\widetilde{\pi(\Sigma)}$ on $\Om 1G$. We claim that
  these relations are equal, under our identification of underlying
  sets with $\hat{\mathbb Z}$. Since $\hat{\mathbb Z}$ is a closed
  ideal in $\hat{\mathbb N}$ and $\tilde{\Sigma}$ is a closed
  congruence on $\hat{\mathbb N}=\Om 1{Com}$, the relation $\rho$ is a
  closed congruence on $\hat{\mathbb Z}$. We have
  $\pi(\Sigma)\subseteq \Sigma_0$, because each pseudoidentity
  $x^{\omega+\lambda}=x^{\omega+\mu}$ is provable from
  $x^\lambda=x^\mu$. Hence, $\pi(\Sigma)\subseteq \tilde{\Sigma}$,
  where $\pi(\Sigma)$ is a relation on~$\hat{\mathbb Z}$. Therefore,
  $\pi(\Sigma)\subseteq \rho$ and the transitive-topological closure
  $\tau=\widetilde{\pi(\Sigma)}$ is also a subset of the closed
  congruence $\rho$. The reverse inclusion $\rho\subseteq \tau$ will
  be proved if we establish (by the induction), for each $\alpha$, the
  inclusion $\Sigma_\alpha |_{\hat{\mathbb Z}\times \hat{\mathbb Z}}
  \subseteq (\pi(\Sigma))_\alpha$.

  Let $\alpha=0$ and $(x^\lambda,x^\mu)\in \Sigma_0$, where
  $\lambda,\mu\in\hat{\mathbb Z}$. Taking into account 
  commutativity, we may assume that $x^\lambda=x^{a+bk}$,
  $x^\mu=x^{a+b\ell}$, where $a,b,k,\ell\in \hat{\mathbb N}$ and
  $(x^k=x^\ell) \in \Sigma$. Then, $(x^{\omega+k}=x^{\omega+\ell})\in
  \pi(\Sigma)$ and we have $(x^{a+b(\omega+k)}=x^{a+b(\omega+\ell)}) \in
  (\pi(\Sigma))_0$. Since $\lambda,\mu\in\hat{\mathbb Z}$, we see that
  $a+b(\omega+k)=a+bk+\omega=\lambda+\omega=\lambda$ and, similarly,
  $a+b(\omega+\ell)=\mu$, which yields $\Sigma_0 |_{\hat{\mathbb Z}\times
    \hat{\mathbb Z}} \subseteq (\pi(\Sigma))_0$.

  Now assume that $\Sigma_{2\gamma} |_{\hat{\mathbb Z}\times
    \hat{\mathbb Z}} \subseteq (\pi(\Sigma))_{2\gamma}$ and let
  $(x^\lambda=x^\mu)\in \Sigma_{2\gamma+1}$, where
  $\lambda,\mu\in\hat{\mathbb Z}$. Then there is a finite sequence
  $\lambda=\lambda_0,\lambda_1, \dots ,\lambda_m=\mu\in\hat{\mathbb
    N}$ such that
  $(x^{\lambda_{i-1}}=x^{\lambda_{i}})\in\Sigma_{2\gamma}$, for each
  $i=1,\dots ,m$. Since $\Sigma_{2\gamma}$ is stable under
  multiplication, we deduce that
  $(x^{\lambda_i+\omega}=x^{\lambda_{i+1}+\omega})\in\Sigma_{2\gamma}|_{\hat{\mathbb
      Z}\times \hat{\mathbb Z}} $. By the induction assumption, these
  pairs also belong to $(\pi(\Sigma))_{2\gamma}$ and since
  $\lambda+\omega=\lambda$ and $\mu+\omega=\mu$, we obtain
  $(x^{\lambda}=x^{\mu})\in (\pi(\Sigma))_{2\gamma+1}$.

  Let $\Sigma_{2\gamma+1} |_{\hat{\mathbb Z}\times \hat{\mathbb Z}}
  \subseteq (\pi(\Sigma))_{2\gamma+1}$ and assume that
  $(x^\lambda=x^\mu)\in \Sigma_{2\gamma+2}$, where
  $\lambda,\mu\in\hat{\mathbb Z}$. Then, there is an infinite sequence
  of pseudoidentities from $\Sigma_{2\gamma+1}$ converging to
  $x^\lambda=x^\mu$. Multiplying the pseudoidentities in the sequence
  by $x^\omega$ and using the fact that $\Sigma_{2\gamma+1}$ is a
  congruence, we obtain the sequence of pseudoidentities from
  $\Sigma_{2\gamma+1} |_{\hat{\mathbb Z}\times \hat{\mathbb Z}}$ with
  the limit $x^\lambda=x^\mu$. Now, using the induction hypothesis we
  obtain, that $(x^{\lambda}=x^{\mu})\in (\pi(\Sigma))_{2\gamma+2}$.

  To accomplish the proof of the claim by transfinite induction we
  need to consider a limit ordinal $\alpha$, but the inclusion follows
  immediately from the inclusions for smaller ordinals.

  Now, we are ready to finish the proof of the theorem. We assumed
  that $\op\Sigma\cl$ satisfies the pseudoidentity
  $x^\varepsilon=x^\delta$, where $\varepsilon,\delta\in \hat{\mathbb
    Z}$. We need to show that $(x^\varepsilon=x^\delta)\in
  \tilde{\Sigma}$, or equivalently written $(\varepsilon,\delta)\in
  \rho$. Assume on the contrary that $(\varepsilon,\delta)\not\in
  \rho$. We have proved the equality of two closed congruences $\rho$
  and $\tau$. So, we can consider the monoid $\hat{\mathbb
    Z}/\rho=\hat{\mathbb Z}/\tau$, which is a profinite monoid,
  because it is a quotient of $\Om 1G$ by the relation
  $\tau=\tilde{\Gamma}$ for $\Gamma=\pi(\Sigma)$ and $\pv G$ is strong
  by Corollary~\ref{c:conjecture-group-case}. Thus, there is a
  continuous homomorphism $f : \hat{\mathbb Z} \rightarrow G$ onto a
  finite (cyclic) group $G\in \op \pi(\Sigma) \cl_{\pv G} \subseteq
  \op\Sigma\cl $, such that
  $f(\varepsilon)\not=f(\delta)$. Now, we can consider the composition
  of $f$ with $\pi$ and we obtain the continuous homomorphism $g:\Om
  1{Com} \rightarrow G$, $g(x^\lambda)=f(\lambda+\omega)$, such that
  $g(x^\varepsilon)\not=g(x^\delta)$. This is a contradiction because
  the group $G\in \op\Sigma\cl$ must satisfy the pseudoidentity
  $x^\varepsilon=x^\delta$. Therefore, we have
  $(x^\varepsilon=x^\delta)\in \tilde{\Sigma}$ and the proof is
  finished.
\end{proof}

To complete the program already followed in Sections~\ref{sec:groups}
and~\ref{sec:CS}, we establish missing strongness property for the
pseudoidentity $xy=yx$. It was not established earlier, namely as part
of Proposition~\ref{p:t-strong}, because our argument depends on
Corollary~\ref{c:H}.

\begin{Prop}
  \label{p:Com-t-strong}
  The pseudoidentity of monoids $xy=yx$ is t-strong.
\end{Prop}

\begin{proof}
  Let $\Sigma$ be a set of \pv M-pseudoidentities such that all
  monoids in $\op\Sigma\cl$ are commutative. We need to show that
  $\Sigma$ proves~$xy=yx$.

  In the terminology introduced at the end of
  Section~\ref{sec:t-strong}, by a theorem of Margolis and Pin
  \cite{Margolis&Pin:1984c}, a complete set of excluded monoids for
  the pseudoidentity $xy=yx$ is given by all non-Abelian groups
  together with $N^1$, $B(1,2)^1$, and $B(2,1)^1$. Hence, each such
  monoid fails some pseudoidentity from~$\Sigma$. Actually, since
  $xy=yx$ entails $x^\omega y=yx^\omega$ and the latter pseudoidentity
  is t-strong by Proposition~\ref{p:t-strong}(\ref{item:ZE-t-strong}),
  we already know that $\Sigma$ proves $x^\omega y=yx^\omega$.

  On the other hand, as $N^1$~fails some pseudoidentity $u=v$
  from~$\Sigma$, either $\Sigma$ proves $x^\omega=1$ or $N^1$ fails a
  two-variable pseudoidentity provable from $\Sigma$ by evaluating $x$
  as~$a$ and $y$ as~$b$. Taking into account that the only nonzero
  product involving at least two factors using either $a$, $b$, or
  both is $ab$, we deduce that $\Sigma$ proves either $x=x^{\omega+1}$
  or a nontrivial pseudoidentity of the form $xy=w$. If $w$ is not
  $yx$, then one may substitute one of the variables by~$1$ to get a
  nontrivial pseudoidentity of the from $x=x^\alpha$, which again
  entails $x=x^{\omega+1}$ by Lemma~\ref{l:1-var-pseudoidentities}.
  Hence, $\Sigma$ deduces either $x=x^{\omega+1}$ or $xy=yx$, which
  means that we may as well assume that $\Sigma$ proves the former
  pseudoidentity. Thus, $\Sigma$ proves the pseudoidentities
  \begin{align}
    \label{eq:CR}
    x^{\omega+1}&=x \\
    \label{eq:ZE}
    x^\omega y&=yx^\omega,
  \end{align}
  which are known to define the join $\pv{Sl}\vee\pv G$,
  that is, the pseudovariety of all finite monoids that are
  semilattices of groups \cite[Exercise~9.1.4]{Almeida:1994a}. Since
  it is not so easy to deduce pseudoidentities from $\Sigma$ from the
  knowledge that non-Abelian groups fail some pseudoidentity
  in~$\Sigma$, we proceed instead to reduce our problem to the group
  case and apply Corollary~\ref{c:H} to draw the conclusion that
  $\Sigma$ proves $xy=yx$.

  First, we exhibit other provable consequences of $\Sigma$. Since we
  deal with the pseudoidentity $xy=yx$, we work over the alphabet
  $A=\{x,y\}$ only, even though more general pseudoidentities may be
  handled similarly. In fact we list some useful pseudoidentities
  that are provable from~\eqref{eq:CR} and~\eqref{eq:ZE}.

  We have $(xy)^\omega= x^\omega (xy)^\omega y^\omega=y^\omega
  (xy)^\omega x^\omega$, where we used \eqref{eq:CR} first and
  \eqref{eq:ZE} in the second step. Then we get $y^\omega (xy)^\omega
  x^\omega=y^{\omega-1}(yx)^{\omega+1} x^{\omega-1}= y^{\omega-1} yx
  x^{\omega-1}=y^\omega x^\omega=x^\omega y^\omega$, where the first
  and third equality hold in $\Om AM$ and the second and fourth
  equality follow respectively from \eqref{eq:CR} and~\eqref{eq:ZE}.
  Now, for each word $w\in A^*$ containing both variables $x$ and $y$,
  if we use the pseudoidentities $(xy)^\omega =x^\omega y^\omega$ and
  \eqref{eq:ZE} repeatedly, we may prove $w^\omega=x^\omega y^\omega$.
  Hence, from \eqref{eq:CR} and~\eqref{eq:ZE} we may prove $w^\omega
  x^\omega y^\omega= x^\omega y^\omega$ for every pseudoword $w$ over
  the alphabet $\{x,y\}$.

  Recall that $\tilde{\Sigma}$ is a relation on $\Om AM$. Let
  $\Gamma=\Sigma\cup\{x^\omega=1\}$ and consider $\tilde{\Gamma}$ on
  the same monoid $\Om AM$. Then $\op\Gamma\cl= \op\Sigma\cl\cap \pv
  G$. Applying Corollary~\ref{c:H} we obtain that $(xy=yx)\in
  \tilde{\Gamma}$. We claim, for any pseudoidentity $u=v$, that
  $(u=v)\in \tilde{\Gamma}$ implies $(ux^\omega y^\omega=vx^\omega
  y^\omega)\in \tilde{\Sigma}$. This gives the proof of the statement
  because the pseudoidentities $xy=xyx^\omega y^\omega$ and $yx=yx
  x^\omega y^\omega$ are provable from \eqref{eq:CR} and~\eqref{eq:ZE}.

  The claim will be proved if we show, for every ordinal $\alpha$, the
  following implication:
  $$(u=v)\in \Gamma_\alpha \implies (ux^\omega y^\omega=vx^\omega
  y^\omega)\in \tilde{\Sigma}\, .$$

  Let $\alpha=0$ and $(u=v)\in \Gamma_0$. If we use in the proof of
  $u=v$ a pseudoidentity from $\Sigma$, then $(u=v)\in \Sigma_0$ and
  therefore also $(ux^\omega y^\omega=vx^\omega y^\omega) \in
  \Sigma_0$. So, we may assume, without loss of generality, that
  $u=\mathbf{t}(\varphi(x^\omega),w_1,\ldots,w_n),
  v=\mathbf{t}(\varphi(1),w_1,\ldots,w_n)$, where $\varphi:\Om
  1M\to\Om AM$ is a continuous homomorphism, $\mathbf{t}$ is a word
  and $w_i\in\Om AM$ ($i=1,\ldots,n$). Using \eqref{eq:ZE} and the
  fact that $(\varphi(x))^\omega x^\omega y^\omega=x^\omega y^\omega$
  is $\Sigma$-provable, we get that $(ux^\omega y^\omega=vx^\omega
  y^\omega) \in \tilde{\Sigma}$.

  The other steps of the induction proof are easy to see as
  $\tilde{\Sigma}$ is a closed congruence.
\end{proof}

As in Sections~\ref{sec:groups} and~\ref{sec:CS}, we may now apply
Proposition~\ref{p:transfer-up} to obtain the following result.

\begin{Cor}
  \label{c:Com}
  For $\pv U=\pv M$, every set $\Sigma$ of pseudoidentities defining a
  pseudovariety of commutative monoids is h-strong.\qed
\end{Cor}

\section{Conclusion}
\label{sec:conclusion}

We have introduced a natural and sound proof scheme for
pseudoidentities. We have given ample evidence for the conjecture that
our proof scheme is complete. There is a nice connection with the much
studied notions of reducibility and tameness. In fact, in all our
examples built on reducible pseudovarieties, we have shown that
$\Sigma_n=\tilde{\Sigma}$ for some integer $n$. But, in
Proposition~\ref{p:J-notSigma2}, we also established that $n=2$ is not
enough. It is expectable that in general there may be no $n\ge1$ such
that $\Sigma_n=\tilde{\Sigma}$, or perhaps even
$\Sigma_\omega\ne\tilde{\Sigma}$, but we have found no example in
which that is the case.

We have not tried to treat exhaustively all cases of pseudovarieties
which are known to be tame. Those that we have considered are perhaps
those that appear more frequently in the literature. Of course, since
we believe in our conjecture, more such results would only be adding
evidence to our claim that we have identified a completely general
phenomenon.

It is the knowledge of congruences on suitable relatively profinite
monoids or semigroups that allowed us to prove that the
pseudovarieties \pv G, \pv{CS}, and \pv{Com} are strong. Much is also
known about the congruences on~\Om A{CR} but we have not taken the
natural path of trying to prove that the pseudovariety \pv{CR} is
strong.

\section*{Acknowledgments}

The first author acknowledges partial funding by CMUP
(UID/MAT/ 00144/2013) which is funded by FCT (Portugal) with national
(MATT'S) and European structural funds (FEDER) under the partnership
agreement PT2020. %
The second author was supported by Grant 15-02862S of the Grant Agency
of the Czech Republic.

\bibliographystyle{amsplain}
\bibliography{sgpabb,ref-sgps}

\end{document}